\documentclass[11pt]{article}
\usepackage[a4paper]{anysize}\marginsize{3.5cm}{3.5cm}{1.3cm}{2cm}
\pdfpagewidth=\paperwidth \pdfpageheight=\paperheight  
\usepackage{amssymb,amsmath}
\usepackage{hyperref}
\usepackage{texdraw}
\usepackage{xypic}
\long\def\eatit#1{}

\newtheorem{satz}{Satz}[subsection]

\newtheorem{corollary}[satz]{Corollary}
\newtheorem{definition}[satz]{Definition}
\newtheorem{example}[satz]{Example}

\newtheorem{lemma}[satz]{Lemma}
\newtheorem{proposition}[satz]{Proposition}
\newtheorem{remark}[satz]{Remark}

\newtheorem{problem}[satz]{Problem}

\newtheorem{theorem}[satz]{Theorem}
\newtheorem{conjecture}[satz]{Conjecture}

\newenvironment{proof}[1][Proof]{\trivlist\item[\hskip\labelsep{\textit{#1.}}]}{\hspace*{\fill}$\Box$\endtrivlist}

\renewcommand\emptyset{\varnothing}  
\renewcommand\ge{\geqslant}  
\renewcommand\le{\leqslant}  
\renewcommand\geq{\geqslant}  
\renewcommand\leq{\leqslant}  
\renewcommand\epsilon{\varepsilon}
\renewcommand\phi{\varphi}
\renewcommand\bar{\overline}
\renewcommand\hat{\widehat}
\renewcommand\tilde{\widetilde}
\renewcommand\O{{\mathcal O}}
\renewcommand\P{\mathbb P}
\newcommand\engqq[1]{``#1''}

\newcommand\be{\begin{eqnarray*}}
\newcommand\ee{\end{eqnarray*}}

\newcommand\eqnref[1]{(\ref{#1})}

\newcommand\eps{\varepsilon}
\newcommand\tensor{\otimes}
\newcommand\N{\mathbb N}
\newcommand\Q{\mathbb Q}
\newcommand\R{\mathbb R}
\newcommand\Z{\mathbb Z}
\newcommand\A{\mathbb A}
\newcommand\C{\mathbb C}
\newcommand\Osc{\mathbb O}

\newcommand\mult{{\rm mult}}

\newcommand\newop[2]{\newcommand#1{\mathop{\rm #2}\nolimits}}
\newop\Sym{Sym}

\newcommand\OO{{\mathcal O}}

\newcommand\calo{\OO}

\newcommand{\HH}[3]{H^{{#1}} \big( {#2} , {#3}
\big) }
\newcommand{\hh}[3]{h^{{#1}} \big( {#2} , {#3}
\big) }

\newcommand{\equ}{\ensuremath{ \,=\, }}
\newcommand{\dsubseteq}{\ensuremath{\,\subseteq\,}}
\newcommand{\dleq}{\ensuremath{\,\leq\,}}
\newcommand{\dgeq}{\ensuremath{\,\geq\,}}

\newcommand{\olc}{\ensuremath{\Omega_X^1(\log C)}}

\newtheorem{thevarthm}[satz]{\varthmname}

\newcommand\set[1]{\left\{#1\right\}}
\newcommand\with{\,\,\vrule\,\,}

\newcommand\Pic{\mbox{\rm Pic}}
\newcommand\Bl{\mbox{\rm Bl}}

\newcommand{\shf}{\ensuremath{{\mathcal F}}}
\newcommand{\shi}{\ensuremath{{\mathcal I}}}
\newcommand{\shl}{\ensuremath{{\mathcal L}}}
\newcommand\cali{\shi}
\DeclareMathOperator{\rk}{rank}
\newcommand{\ses}[3]{\ensuremath{0\rightarrow #1 \rightarrow #2 \rightarrow #3 \rightarrow 0}}
\newcommand{\st}[1]{\ensuremath{ \left\{ #1 \right\} }}

\DeclareMathOperator{\bight}{bight}
\DeclareMathOperator{\lct}{lct}
\newcommand{\m}{\mathfrak{m}}
\newcommand{\J}{\mathcal{J}}

\begin{document}

\title{Recent developments and open problems\\ in linear series}
\author{Th. Bauer, C. Bocci, S. Cooper,
S. Di Rocco, M. Dumnicki, \\ B. Harbourne, K. Jabbusch,
 A. L. Knutsen, A.~K\"uronya,
R. Miranda, \\ J. Ro\'e, H. Schenck,
T. Szemberg, Z. Teitler}
\date{April 16, 2011}

\maketitle
\thispagestyle{empty}
\pagestyle{myheadings}
\markright{\hfill\small\sl Recent developments and open problems in Linear Series\hfill}

\tableofcontents


\section{Introduction}

  In the week of  October 3--9, 2010, the Mathematisches Forschungsinstitut at Oberwolfach
   hosted the Mini-Workshop \engqq{Linear Series on Algebraic Varieties.}
   These notes contain a variety of interesting problems which motivated
   the participants prior to the event, and examples, results and further problems
   which grew out of discussions during and shortly after the workshop.
   Many arguments presented here are scattered in the literature or constitute
   \engqq{folklore.} It was one of our aims to have a usable and easily accessible collection
   of examples and results.\footnote{Cooper's participation was supported
   by the ``US Junior Oberwolfach Fellows'' joint NSF-MFO
   program under NSF grant DMS-0540019.
   Partial support during this project is kindly acknowledged as follows:
   Bocci by Italian PRIN funds;
   Di Rocco by Vetenskaps\aa det�s grant NT:2006-3539;
   K\"uronya by  DFG-Forschergruppe 790 ``Classification of Algebraic Surfaces
   and Compact Complex Manifolds'' and the OTKA grants 77476 and 77604 by the Hungarian Academy of Sciences;
   Ro\'e by Spanish Ministerio de Educaci\'on y Ciencia,
   grant MTM2009-10359;
   Schenck by NSF 07--07667, NSA 904-03-1-0006;
   Szemberg by MNiSW grant N N201 388834}

\section{Original problems}
   We begin with a list of problems which were suggested by the participants for the Mini-Workshop.
   This list was discussed for three months before the workshop began.

\subsection{Asymptotic effectivity (B. Harbourne)}

Let $S=\left\{p_1,\dots,p_r\right\}$ be distinct points in $\P^N$, over an algebraically closed ground field $k$
of arbitrary characteristic.
Let $f:X\to\P^N$ be the morphism obtained by blowing up $p_1,\dots,p_r$, and denote the exceptional
divisors by $E_1,\dots,E_r$. Let $H=f^*(\O_{\P^N}(1))$
and let $L(d,m)=dH-m(E_1+\dots +E_r)$.
Waldschmidt \cite{refW77} introduced and showed the existence of the following quantity:
$$e(S):=\lim\limits_{m\to\infty}\frac{a(S,m)}{m}$$
where $a(S,m):=\min\left\{d:\, h^0(X,L(d,m))>0\right\}$.
It follows from the proof that $me(S)\leq a(S,m)$ for all $m\geq 1$.

\begin{problem}
Develop computational or conceptual methods for evaluating, estimating or bounding $e(S)$.
\end{problem}

It is trivial to show that $a(S,m)=rm$ and hence $e(S)=r$ when $N=1$.
It is an open problem in general to compute $e(S)$ when $N>1$.
For example, for $r>9$ generic points $p_i\in\P^2$, a still open conjecture of Nagata \cite[Conjecture, p. 772]{refN59}
is equivalent to having $e(S)=\sqrt{r}$.

Using complex analytic methods, Waldschmidt \cite{refW77} and Skoda \cite{refSk77}
showed that
\begin{equation}\label{WSbound}
\frac{a(S,1)}{N}\leq e(S).
\end{equation}
Given $k\geq0$, by \cite[Theorem 1.1(a)]{refHH02}, for any finite subset
$S\subset \P^N$ and all $m\ge1$ we have
$$\frac{a(S,k+1)}{k+N}\leq\frac{a(S,m(N+k))}{m(k+N)}$$
and hence
$$\frac{a(S,k+1)}{k+N}\leq e(S)\leq \frac{a(S,k+1)}{k+1}.$$
Taking $k=0$ recovers the Waldschmidt-Skoda bound \eqref{WSbound}.
Taking large values of $k$ gives one a way of computing
arbitrarily accurate estimates of $e(S)$ \cite{refHR03}, but computing
$a(S,k)$ for large $k$ is difficult to do.

Chudnovsky \cite{refCh81} conjectured (and proved for $N=2$) the stronger bound
\begin{conjecture}\label{chudnovsky conjecture}
$$\frac{a(S,1)+N-1}{N}\leq e(S).$$
\end{conjecture}
\noindent(For the proof when $N=2$, reduce to the case that $r=\binom{a+1}{2}$, where $a=a(S,1)$,
and use the fact that then $L(a,1)$ is nef.)

Two examples are known where equality in \eqnref{chudnovsky conjecture} holds: when
the points lie in a hyperplane, or when $S$ is a {\em star configuration} \cite{refBH10} (i.e.,
given a set of $s$ hyperplanes in $\P^N$
such that at most $N$ of these hyperplanes meet at any single point,
$S$ is the set of $r=\binom{s}{N}$ points at which $N$ of the hyperplanes meet).

\begin{problem}
If for some $m$ we have
$$\frac{a(S,m)}{m}=\frac{a(S,1)+N-1}{N}$$
is it true that $S$ is either contained in a hyperplane or is a star configuration?
What if $$e(S)=\frac{a(S,1)+N-1}{N}\ ?$$
\end{problem}

Bocci and Chiantini \cite{refBC10} show for $N=2$ that
$\frac{a(S,2)}{2}=\frac{a(S,1)+1}{2}$ implies
that $S$ is either a set of collinear points or a star configuration.

\subsection{Semi-effectiveness (B. Harbourne)}

Again let the ground field $k$ be an algebraically closed field
of arbitrary characteristic.

\begin{definition}\label{semieff}
   Let $X$ be an algebraic variety and $L$ a line bundle on $X$. We say that $L$ is \emph{semi-effective},
   if there exists $n>0$ such that $h^0(nL)>0$.
\end{definition}
   Let $p_1,\dots,p_r$ be distinct points in $\P^N$.
   Let $f:X\to\P^N$ be obtained by blowing up $p_1,\dots,p_r$ with the exceptional
   divisors being $E_1,\dots,E_r$.
   Let
   $$D=dH-m_1E_1-\dots-m_rE_r,$$
where $H$ is the pullback via $f$ of a general hyperplane.

The following question was raised by M. Velasco and D. Eisenbud (in an email from Velasco to Harbourne, November 2009).
\begin{problem}\label{VEprob}
Is there a way to determine if $D$ is semi-effective?
\end{problem}

   For a specific problem consider
   $$D=13L-5E_1-4E_2-\dots-4E_{10}$$
   for \emph{generic} points $p_1,\dots,p_{10}\in\P^2$.
   M. Dumnicki and J. Ro\'e have independently shown that this system
   is not semi-effective (see Section \ref{sec:geometrization of Dumnicki}).
   But there are infinitely many more similar examples
   for which semi-effectivity is still not known. For example,
   given generic points $p_i\in\P^2$, consider $D=111L- 36E_1-35E_2-\cdots- 35E_{10}$ or
   more generally $D=dL-(m+1)E_1-mE_2-\cdots-mE_{10}$, where
   $d=(b^2+a^2)/2$, $m+1=ba$, $m=(b^2-a^2)/6$ and where $0<a< b$ are odd integers
   satisfying $(a+3b)^2-10b^2=6$.
Note that $D^2=0$ but $D$ cannot be reduced by Cremona transformations.
According to the SHGH Conjecture \cite{refS61, Ha86, Gi87, Hi89},
we expect that none of these $D$ are semi-effective.

Limits like those of Waldschmidt are relevant to Problem \ref{VEprob}. Let $a(\sum_{i=1}^rm_ip_i)$
be the least $t$
such that $h^0(tL-\sum_im_iE_i)>0$ and define
$$e(\sum_im_ip_i):=r\lim_t \frac{a(t\sum_im_ip_i)}{t\sum_im_i}.$$
Then $D=dL-\sum_{i=1}^rm_i$ is semi-effective if $e(\sum_im_ip_i)/r<d/\sum_im_i$, and it is not
semi-effective if $e(\sum_im_ip_i)/r>d/\sum_im_i$. It is not clear
whether $D$ is semi-effective when $e(\sum_im_ip_i)/r=d/\sum_im_i$.
But, if the SHGH Conjecture is true, this boundary case is precisely the situation of the examples
in the preceding paragraph.

\subsection{Stability of speciality (T. Szemberg)}\label{specstab}

   Let $p_1,\dots,p_r$ be points in the projective plane.
   Let $f:X\to\P^2$ be the blow-up of $p_1,\dots,p_r$ with exceptional
   divisors $E_1,\dots,E_r$. Let $H=f^*(\O_{\P^2}(1))$, and
   let
   $$D=dH-m_1E_1-\dots-m_rE_r.$$
   Assume that the divisor $D$ is special (i.e., $D$ is effective with $h^1(X, D)>0$). Is it then true
   that
   $$nD \mbox{ is special for all}\, n>1?$$

   A somewhat more demanding problem is to determine whether the
   asymptotic cohomology function $\hat{h}^1$ as defined in \cite{Kur06} is positive.
   This is not true if the points are arbitrary (or less than $10$?).
   Indeed, let $p_1,\dots,p_9$ be intersection points of two cubics
   and let $L=3H-\sum_{i=1}^9E_i$ be the anti-canonical pencil.
   Then
   $$h^0(nL)=n+1\;\mbox{ and }\, h^1(nL)=n\;\mbox{ and }\, h^2(nL)=0,$$
   so that all $nL$ are special but all asymptotic cohomology functions vanish.

\subsection{Regularity for generic monomial zero schemes (J. Ro\'e)}

Let $I \subset k[x,y]$ be a monomial ideal, and let $Z={\rm Spec}\, k[x,y]/I$.
Let $n=\dim_k k[x,y]/I$ be the length of $Z$, which we assume to be finite
(in this case, $I$ is $(x,y)$-primary, and $Z$ is supported at the origin).
For each (irreducible smooth projective) surface $S$ defined over $k$ there
is an irreducible constructible subset ${\rm Hilb}_I S$
of ${\rm Hilb}^n S$ whose closed points are the subschemes
$Y$ of $S$ isomorphic to $Z$.

Of course, if $I=(x,y)^m$ then $Z$ is just an $m$-fold point, and
${\rm Hilb}_I S\cong S$ is the set of points of $S$ taken
with multiplicity $m$. If $I=(y^2,yx^2,x^3)$ then $Z$ is a
\emph{cusp scheme}, which means that curves containing $Z$
have at least a cusp at the point supporting $Z$ (i.e., they
have a cusp or a more special singularity, and generically
it is a cusp). Since the cusp scheme marks the tangent direction
to the cusp, ${\rm Hilb}_I S$ is in this case naturally isomorphic
to the (projectivized) tangent bundle of $S$.
Other monomial ideals correspond to other singularity
``types.''

\begin{problem}
Describe ${\rm Hilb}_I S$. Is it locally closed? What  adjacencies are there
between such subsets of the Hilbert scheme?
\end{problem}

As ${\rm Hilb}_I S$ is irreducible, it makes sense to consider general
(or very general, or generic) subschemes $Y$ of $S$ isomorphic to $Z$.
For each divisor $D$ we have an exact sequence
$$ 0 \rightarrow \mathcal{I}_Y \tensor \mathcal{O}_S(D) \rightarrow
\mathcal{O}_S(D) \rightarrow \mathcal{O}_Y(D) \rightarrow 0$$
inducing the usual exact sequence in cohomology.

\begin{problem}
\label{pro:mon}
What can we say (or conjecture) about the cohomology of
$\mathcal{I}_Y \tensor \mathcal{O}_S(D)$?
To be more precise, it would be nice to have conditions on $S$, $I$ and $D$
implying that $h^0=\max \{0,h^0(\mathcal{O}_S(D))-n\}$,
in which case $h^1$ and $h^2$ are easily computed by Riemann-Roch and
duality.
\end{problem}

Assuming the previous problem is understood, one can then
further ask about the base locus of global sections of
$\mathcal{I}_Y \tensor \mathcal{O}_S(D)$, and ask if they
cut out the scheme $Z$.
Such questions are of interest in the construction of curves with
imposed singularities, and have been studied mainly for schemes
of small multiplicity (the multiplicity of $Z$ is the maximum $m$
such that $I \subset (x,y)^m$). Note that multiplicity 1 schemes
are curvilinear and well known; multiplicity 2 monomial schemes are
also quite well understood \cite{Roe06}.
The same kind of questions arise as auxiliary problems for the
induction arguments of differential Horace methods \cite{Eva07},
even when one is primarily interested
in linear series defined by ordinary multiple points.

Note also that the scheme defined by  $I=(y,x^r)^m$ is a specialization
(or collision) of $r$ distinct $m$-fold points; thus the
dimension of the linear series $H^0(\mathcal{I}_Y \tensor \mathcal{O}_S(D))$
is by semicontinuity a bound for the dimension of the linear series
determined by a set of $r$ general $m$-fold points. This gives a link
with the Nagata conjecture \cite{refN59} and the
SHGH conjecture \cite{refS61, Ha86, Gi87, Hi89}.
More precisely:

\begin{conjecture}\label{monomhh}
Let $S=\mathbb{P}^2$ and $I=(y,x^r)^m$, where $r$ and $m$ are natural
numbers with $r>9$. Then for all $d> 0$, and general $Y\in {\rm Hilb}_I S$,
$$H^0(\mathcal{I}_Y \tensor \mathcal{O}_{\mathbb{P}^2}(D)) \equ \max \left\{0, \binom{d+2}{2}-r\binom{m+1}{2} \right\} .$$
\end{conjecture}

\begin{conjecture}\label{monomnag}
Let $S=\mathbb{P}^2$ and $I=(y,x^r)^m$, where $r$ and $m$ are natural
numbers with $r>9$. Then for all $d \le m\sqrt r$, and general $Y\in {\rm Hilb}_I S$,
$H^0(\mathcal{I}_Y \tensor \mathcal{O}_{\mathbb{P}^2}(D))=0$.
\end{conjecture}

Conjecture \ref{monomhh} implies the uniform Harbourne-Hirschowitz
conjecture, and Conjecture \ref{monomnag} implies Nagata's
conjecture. It is also clear that Conjecture \ref{monomhh} implies
Conjecture \ref{monomnag}.

A conjecture in terms of monomial ideals implying the general
Harbourne-Hirschowitz conjecture (without uniformity assumptions
on the multiplicities) can be stated similarly; we skip it here to avoid
introducing the necessary notations, which would lengthen this section
unnecessarily, and refer to Hirschowitz's description of the ``collisions
de front'' in \cite{Hir85} instead.

\subsection{Bounding cohomology (B. Harbourne)}

There are various equivalent versions of the SHGH Conjecture \cite{refS61, Ha86, Gi87, Hi89}. Here's one:
\begin{conjecture}[SHGH]
   Let $C\subset X$ be a prime divisor where
   $X\to\P^2$ is the blow-up of generic points $p_1,\ldots,p_s$.
   Then $h^1(X,\OO_X(C))=0$.
\end{conjecture}
\begin{problem}
   How can we remove the assumption about the points
   being generic in the SHGH Conjecture?
\end{problem}
The following conjecture arose out of discussions between Harbourne, J. Ro\'e, C. Ciliberto and R. Miranda.
[NB: Corollary \ref{exp1} gives a counterexample. See also
Proposition \ref{prop:SHGHtoric}.]

\begin{conjecture}\label{relating hs}
   Let $X$ be a smooth projective surface (either rational
   or assume the characteristic is 0). Then there exists a constant $c_X$
   such that for every prime divisor $C$ we have
   $h^1(X, C)\leq c_Xh^0(X,C)$.
\end{conjecture}
The SHGH Conjecture is that $c_X=0$ when $X$ is obtained by blowing up generic points of $\P^2$.

\subsection{Algebraic fundamental groups and Seshadri numbers (J.-M. Hwang)}

  Denote by $\hat{\pi}_1(Y)$ the algebraic fundamental group
   of an irreducible variety $Y$. Following \cite[Definition (2.7.1)]{Kol93},
   we say that a projective manifold $X$ has large algebraic
   fundamental group if for every irreducible variety $Z \subset X$
   and its normalization $\nu: \bar{Z} \rightarrow X$, the image of
   the induced homomorphism on the algebraic fundamental groups
   $$\nu_*: \hat{\pi}_1(\bar{Z}) \rightarrow \hat{\pi}_1(X)$$ is
   infinite. This is equivalent to saying that the algebraic
   universal cover of $X$ does not contain a complete subvariety.
   The proof of  \cite[Lemma 8.2]{Kol93}   gives the following.

\begin{proposition}
   Let  $N$ be a positive number. Let $X$ be a
   projective manifold with large algebraic fundamental group and let
   $L$ be an ample line bundle on $X$. Then there exists a finite
   \'etale cover $p : X' \rightarrow X$ such that any irreducible
   subvariety $W$ in $X'$ satisfies $(p^*L)^{\dim (W)} \cdot W >N$.
\end{proposition}
One can then ask the following:
\begin{problem}
   Let $X$ be a projective manifold with large
   algebraic fundamental group and let $L$ be an ample line bundle on
   $X$.
   \begin{enumerate}
      \item[(1)] Given a positive number $N$,
         does there exist a finite \'etale cover $p : X' \rightarrow X$ such
         that the Seshadri number of $p^*L$ at any point is bigger than $N$?
      \item[(2)] Given a positive number $N$, does there exist a
         finite \'etale cover $p : X' \rightarrow X$ such that denoting by
         $\tilde{p}: X' \times X' \to X \times X$ the self-product of $p$
         and $D'\subset X' \times X'$ the diagonal, the Seshadri number of
         $\tilde{p}^*(p_1^*L\otimes p_2^*L)$ along $D'$ is bigger than $N$?
      \item[(3)] If the answer to (1) or (2) is negative or unclear,
         what is the condition on the fundamental group of $X$ to guarantee
         a positive answer?
   \end{enumerate}
\end{problem}
\subsection{Blow-ups of $\mathbb{P}^n$ and hyperplane arrangements (H. Schenck)}

\noindent
Let
\[
{\mathcal A} = \bigcup_{i=1}^d V(\alpha_i) \subseteq {\mathbb P}^2
\]
be a union of lines in ${\mathbb P}^2$, $Y$ the singular
locus, and $\pi: X \to {\mathbb P}^2$ the blow-up at $Y$.
Let $R={\mathbb C}[y_1,\ldots,y_d]$, and for each linear dependency
$\Lambda= \sum_{j=1}^k c_{i_j}\alpha_{i_j} =0$ on the lines of ${\mathcal A}$, let
\[
f_\Lambda = \sum_{j=1}^k c_{i_j} (y_{i_1}\cdots \hat
y_{i_{j}} \cdots y_{i_k}).
\]
The ideal $I$ generated by the $f_{\Lambda}$ is called the Orlik-Terao ideal,
and the quotient $C({\mathcal A})=R/I$ is called the Orlik-Terao algebra; of
course, $C({\mathcal A})$ can be defined for arrangements in higher dimensional spaces.
For a real, affine arrangement ${\mathcal A}$, Aomoto conjectured a relationship
between $C({\mathcal A})$ and the topology of ${\mathbb R}^n \setminus {\mathcal A}$,
which Orlik and Terao proved in \cite{OT}.
\begin{example}\rm
Consider ${\mathcal A} = V(x_1x_2x_3(x_1+x_2+x_3)(x_1+2x_2+3x_3)) \subseteq {\mathbb P}^2.$
Since $\alpha_1+\alpha_2+\alpha_3-\alpha_4 = 0$,
\[
y_2y_3y_4+y_1y_3y_4+y_1y_2y_4-y_1y_2y_3 \in I.
\]
The five lines meet in ten points, and every subset of four lines gives
a similar relation, one of which is redundant. Thus, $I$ is generated by
four cubics, which turn out to be the maximal minors of a matrix of
linear forms. This means that $I$ has a
Hilbert-Burch resolution, and $V(I)$ is a surface of degree
six in $\mathbb{P}^4$; a computation shows $V(I)$ has five singular points.
Consider the divisor
\[
D_{{\mathcal A}} = 4E_0 - \!\!\sum\limits_{i=1}^{10}E_i
\]
on $X$, where $X$ is the blow-up of ${\mathbb P}^2$ at the ten points of $Y$,
$E_i$ are the exceptional curves over the singular points,
and $E_0$ is the proper transform of a line. Then $D_{\mathcal A}$ is nef
but not ample; the lines of the original arrangement are
contracted to points, and $I$ is the ideal of $X$ in $Proj(H^0(D_{{\mathcal A}}))$.
\end{example}
The example above is representative of the general case. In \cite{Sch},
it is shown that if
\[
 \phi_{{\mathcal A}}: X \to \mathbb{P}(H^0(D_{{\mathcal A}})^\vee),
\]
then $C({\mathcal A})$ is the homogeneous coordinate ring of $\phi_{{\mathcal A}}(X)$ and $\phi_{{\mathcal A}}$
is an isomorphism on $\pi^*(\mathbb{P}^2 \setminus {\mathcal A}$), contracts the lines of ${\mathcal A}$ to
points, and blows up $Y$.

The motivation for studying $C({\mathcal A})$ arises from its connection to topology.
In \cite{OS}, Orlik and Solomon determined the cohomology ring of
a complex, affine arrangement complement $M= {\mathbb C}^{n+1} \setminus {\mathcal A}$:
$A=H^*(M,\Z)$ is the quotient of the exterior algebra
$E=\bigwedge (\Z^d)$ on generators $e_1, \dots , e_d$
in degree $1$ by the ideal generated by all elements of
the form $\partial e_{i_1\dots i_r}:=\sum_{q}(-1)^{q-1}e_{i_1} \cdots
\widehat{e_{i_q}}\cdots e_{i_r}$, for which
$\mbox{codim }H_{i_1}\cap \cdots \cap H_{i_r} < r$. Since $A$ is a quotient of an exterior algebra, multiplication
by an element $a \in A^1$ gives a degree one differential on
$A$, yielding a cochain complex $(A,a)$:
\[
(A,a)\colon \quad
\xymatrix{
0 \ar[r] &A^0 \ar[r]^{a} & A^1
\ar[r]^{a}  & A^2 \ar[r]^{a}& \cdots \ar[r]^{a}
& A^{\ell}\ar[r] & 0}.
\]
The {\em first resonance variety} $R^1({\mathcal A})$ consists of points
$a=\sum_{i=1}^da_ie_i \leftrightarrow (a_1:\dots :a_d)$ in
$\mathbb{P}(A^1) \cong \mathbb{P}^{d-1}$ for which $H^1(A,a) \ne 0$.  Conjectures
of Suciu \cite{SU} relate the fundamental group of $M$ to $R^1({\mathcal A})$.
Falk showed that $R^1({\mathcal A})$ may be described in terms of combinatorics,
and conjectured that $R^1({\mathcal A})$ is a subspace arrangement, which was
shown by Cohen--Suciu (\cite{CS}) and Libgober--Yuzvinsky (\cite{LY}).
The paper \cite{Sch} describes a connection between combinatorics
of $R^1({\mathcal A})$ which give rise to factorizations of $D_{\mathcal A}$ and
corresponding determinantal equations in $I$.

\begin{problem}
Do the results for lines in $\mathbb{P}^2$ generalize to higher dimension?
For example, if $\mathcal{A} \subseteq \mathbb{P}^n$, then
\cite{Sch} shows that the Castelnuovo-Mumford regularity
of $C({\mathcal A})$ is bounded by $n$. Can an explicit description of the
graded betti numbers of $C({\mathcal A})$ be given in terms of the geometry
of $\mathcal{A}$?
\end{problem}

\subsection{Bounds for symbolic powers (Z. Teitler)}

Let $X$ be a non-singular variety of dimension $n$ defined over the complex numbers
and let $Z \subseteq X$ be a reduced subscheme of $X$ with ideal sheaf $I = I_Z \subseteq \O_X$.
The $p^{th}$ symbolic power of $I$, denoted $I^{(p)}$, is the sheaf of all function germs
vanishing to order $\geq p$ at each point of $Z$.
The inclusion $I^p \subseteq I^{(p)}$ is clear, or in other words $I^r \subseteq I^{(m)}$ for $r \geq m$,
but it is not clear when inclusion in the other direction, $I^{(m)} \subseteq I^r$, holds;
$m \geq r$ is necessary but not sufficient in general.
Related to a result of Swanson~\cite{SW}, Ein--Lazarsfeld--Smith used multiplier ideals
to prove that if every component of $Z$ has codimension $\leq e$ in $X$
then $I^{(re)} \subseteq I^r$ for all $r \in \N$;
in particular $I^{(rn)} \subseteq I^r$ \cite{ELS}.
This was subsequently proved in greater generality by Hochster--Huneke
using the theory of tight closure~\cite{refHH02}.

The big height of a radical ideal $I$, denoted $\bight(I)$, is the maximum codimension of a component of $V(I)$.
Harbourne raised the question whether it is possible to give an improvement of the form
$I^{(m)} \subseteq I^r$ whenever $m \geq f(r)$, for some function $f(r) \leq re$, for all radical ideals of big height $e$.
Bocci--Harbourne~\cite{refBH10} showed if $\gamma > 0$ is such that
for all radical ideals $I$, $I^{(m)} \subseteq I^r$ whenever $m \geq \gamma r$,
then $\gamma \geq n$.
Furthermore, if $I^{(m)} \subseteq I^r$ holds whenever $m \geq \gamma r$
for all radical ideals $I$ with $\bight(I) = e$,
then $\gamma \geq e$.
This shows that the function $f(r) = re$ appearing in the Ein--Lazarsfeld--Smith result
cannot be decreased by lowering the coefficient $e = \bight(I)$.

Harbourne asked whether a constant term can be subtracted, that is whether
$I^{(m)} \subseteq I^r$ holds whenever $m \geq f(r) = er-k$ for all radical ideals of big height $e$, for some $k$.
Values $k \geq e$ do not work (the containment fails if $I$ is a complete intersection and --- rather trivially --- if $r = 1$ or $e=1$).
On the other hand, examples studied by Bocci--Harbourne suggest that $k=e-1$ might work.
Harbourne made the following:
\begin{conjecture}[{\cite[Conjecture 8.4.3]{Primer-Seshadri}}]\label{bound-symbolic-powers-1}
For radical ideals $I$ of big height $e$, $I^{(m)} \subseteq I^r$ whenever $m \geq re-(e-1)$.
\end{conjecture}
A weaker result was observed by Takagi--Yoshida~\cite{takagi-yoshida}
(see also \cite{MR2591736} for an expository account)
who showed that if $\ell < \lct(I^{(\bullet)})$ then $I^{(m)} \subseteq I^r$ whenever $m \geq re - \ell$.
Here $\lct(I^{(\bullet)})$ is the log canonical threshold of the graded system of ideals $I^{(\bullet)}$;
in particular this is always $\leq e$ (so $\ell \leq e-1$).

Harbourne--Huneke have asked if an improvement is possible on the other side of the inclusion:
Instead of asking for $I^{(m)} \subseteq I^r$, they raise the following:
\begin{problem}\label{bound-symbolic-powers-2}
Suppose $(R,\m)$ is a regular local ring of dimension $n$ and $I \subseteq R$ is an ideal with $\bight(I) = e$.
Then do the following hold?
\begin{enumerate}
\item $I^{(m)} \subseteq \m^{rn-r} I^r$ for $m \geq rn$.
\item $I^{(m)} \subseteq \m^{rn-r-(n-1)} I^r$ for $m \geq rn-(n-1)$.
\item $I^{(m)} \subseteq \m^{re-r} I^r$ for $m \geq re$.
\item $I^{(m)} \subseteq \m^{re-r-(e-1)} I^r$ for $m \geq re-(e-1)$.
\end{enumerate}
\end{problem}

\section{Progress}
   In this part we show several solutions to the original problems,
   present examples which are closely related to the problems,
   and provide some evidence either for positive or negative answers
   or forcing reformulation of the original statements.

\subsection{Relating $h^0$ and $h^1$ on surfaces}

   We show that Conjecture \ref{relating hs} is false in general.
   We claim (see Corollary \ref{exp1}) that there exists a surface $X$ such that for an arbitrary
   positive integer $c$ there exists a reduced, irreducible curve $C\subset X$ with
   $$h^1(X,\calo_X(C))>c\cdot h^0(X,\calo_X(C)).$$

   It turns out that an example of Koll\'ar (taken from \cite[Example 1.5.7]{PAG})
   provides a counterexample to Conjecture \ref{relating hs}. We recall briefly the construction,
   in which we will closely follow the exposition mentioned above.
   We consider an elliptic curve $E$ without complex multiplication, and take the abelian surface
\[
Y := E\times E
\]
   to be our starting point. Divisors and the various cones on $Y$ are well understood. The Picard number $\rho(Y)$ equals $3$,
   and $N^1(X)_{\R}$ has the fairly natural basis consisting of the classes of $F_1$, $F_2$
   (the fibres of the two projection morphisms), and that of the diagonal $\Delta\subseteq E\times E$.
   The intersection form on $Y$ is given by the numbers
\[
(F_1^2) \equ (F_2^2) \equ (\Delta^2) \equ 0\ ,
(F_1\cdot F_2) \equ (F_1\cdot \Delta) \equ (F_2\cdot \Delta) \equ 1.
\]
   It is known (for details see for example \cite[Section 1.5.B]{PAG}) that the nef and pseudoeffective cones coincide,
   and a class $C=a_1F_1+a_2F_2+b\Delta$ is nef if and only if $(C^2)\geq 0$ and $(C\cdot H)\geq 0$ for some ample class $H$.
   In coordinates we can express this as
\[
a_1a_2+a_1b+a_2b \,\geq\, 0\ \ \ \text{ and }\ \ \ a_1+a_2+b\,\geq\, 0
\]
   by choosing  the ample divisor $F_1+F_2+\Delta$ for $H$.

   For every integer $n\geq 2$ set
\[
A_n := nF_1 + (n^2-n+1)F_2 - (n-1)\Delta.
\]
   It is immediate to check that
\begin{enumerate}
   \item $(A_n^2) \equ 2$, and
   \item $(A_n\cdot (F_1+F_2)) \equ n^2-2n+3 >0$.
\end{enumerate}

   Koll\'ar now sets $R:= F_1+F_2$, and picks a smooth divisor $B\in |2R|$ (which exists because $2R$ is base point
   free by the Lefschetz theorem \cite[Theorem 4.5.1]{CAV}) to form the double cover $f:X\to Y$ branched
   along $B$. Let $D_n := f^*A_n$.

\begin{proposition}\label{prop:h1 cubic}
   With notation as above,
   $$h^1(X,\calo_X(nD_n))\geq n^3-2n^2+3n-1.$$
\end{proposition}

\begin{proof}
   We will estimate $\hh{1}{X}{\OO_X(nD_n)}$ from below with the help of the
   Leray spectral sequence. It is a standard fact that
\[
E_2^{p,q} := \HH{p}{Y}{R^qf_*\OO_X(nD_n)} \Longrightarrow \HH{p+q}{X}{\OO_X(nD_n)} .
\]
   We are interested in the case $p+q=1$, in which case  the involved terms are
\[
\HH{0}{Y}{R^1f_*(\OO_X(nD_n)} \ \ \ \text{ and }\ \ \ \HH{1}{Y}{f_*\OO_X(nD_n)} .
\]
   Of these the second term survives unchanged to the $E_\infty$ term, and so we obtain an inclusion
\[
\HH{1}{Y}{f_*\OO_X(nD_n)} \hookrightarrow \HH{1}{X}{\OO_X(nD_n)} .
\]
   It is $\HH{1}{Y}{f_*\OO_X(nD_n)}$ that we will estimate from below.
   By \cite[Proposition 4.1.6]{PAG} on the properties of cyclic coverings, one has
\[
f_*\OO_X \equ \OO_Y\oplus \OO_Y(-R)\ ,
\]
   which implies
\[
f_*(\OO_X(nD_n)) \equ \OO_Y(nA_n)\oplus \OO_Y(nA_n-R)
\]
   via the projection formula.

   It follows that
\[
\HH{1}{Y}{f_*(\OO_X(nD_n))} \equ \HH{1}{Y}{\OO_Y(nA_n)}\oplus \HH{1}{Y}{\OO_Y(nA_n-R)}.
\]
   We can determine the second term of the sum from the Riemann--Roch theorem. On the abelian surface
   $Y=E\times E$ one has $\chi(\OO_Y)=0$ and $K_Y=\calo_Y$, hence Riemann--Roch has the particularly simple form
\[
\chi(\OO_Y(nA_n-R)) \equ \frac{1}{2}(nA_n-R)^2\ .
\]

   We compute
\begin{eqnarray*}
(nA_n-R)^2 &  \equ &  n^2 A_n^2 -2n(A_n\cdot R) +R^2 \\
&  \equ &  2n^2-2n(n^2-2n+3)+2  \\
& \equ & -2n^3+6n^2-6n+2 \\
& \equ & -2(n-1)^3 \\
&  < & 0\ ,
\end{eqnarray*}
   therefore neither $nA_n-R$ nor its negative is effective, resulting in
\[
\HH{0}{Y}{\OO_Y(nA_n-R)} \equ 0
\]
and
\[
\HH{2}{Y}{\OO_Y(nA_n-R)} \equ \HH{0}{Y}{\OO_Y(R-nA_n)} \equ 0 .
\]

   Hence we can conclude that
\[
\hh{1}{X}{\OO_X(nD_n)} \geq \hh{1}{Y}{\OO_Y(nA_n-R)} \equ  n^3-2n^2+3n-1 \,>\, 0 .
\]
\end{proof}

\begin{corollary}\label{exp1}
With notation as above, for an arbitrary positive integer $c$ there exists a prime divisor $C$ on $X$ such that
\[
\hh{1}{X}{\OO_X(C)}\, > \, c\cdot \hh{0}{X}{\OO_X(C)}.
\]
\end{corollary}
\begin{proof}
   This is in fact a corollary of the proof of the Proposition \ref{prop:h1 cubic}.
   For $n\geq 2$ the linear system $|nA_n|$ is globally generated
   by the Lefschetz theorem \cite[Theorem 4.5.1]{CAV} and it is not
   composed with a pencil. The same holds true for $|nD_n|$,
   therefore the base-point free Bertini theorem implies that the general element of $|nD_n|$ is reduced and irreducible. Let $C_n$ be such an element. Then
\[
\hh{0}{X}{\OO_X(C_n)} \equ \hh{0}{Y}{\OO_Y(nA_n)} \equ \frac{1}{2}((nA_n)^2) \equ n^2
\]
by the Riemann--Roch theorem. On the other hand,
\[
\hh{1}{X}{\OO_X(C_n)} \geq   n^3-3n^2+3n-1 \, >\, c\cdot n^2
\]
for large enough $n$.
\end{proof}

   The surface $X$ studied above is of general type.
   It still could be true that Conjecture \ref{relating hs} holds when restricted
   to \textit{rational} surfaces, in any characteristic. For some evidence in this direction
   we now prove a particularly strong form of the conjecture in
the case of smooth projective rational surfaces with an
effective anticanonical divisor, and hence in particular for smooth projective toric surfaces.

\begin{proposition} \label{prop:SHGHtoric}  Let $X$ be a smooth projective rational surface having an
effective anticanonical divisor $D\in |-K_X|$.  Then there exists a constant $c_X$ such that for every
prime divisor $C\subset X$ we have $h^1(X, C)\leq c_X$. In fact, the maximum value of
$h^1(X,C)$ is either 0 or 1, or it occurs when $C$ is a component of $D$.
\end{proposition}

\begin{proof}
Let $C$ be a prime divisor on $X$. Clearly, we can assume that $C$ is not a component of $D$.
Thus $-K_X\cdot C=D\cdot C\geq 0$, and if $-K_X\cdot C=0$, then $C$ is disjoint from $D$.

Suppose $-K_X\cdot C>0$. If $C^2\geq0$, then $h^1(X, C)=0$
by \cite[Theorem III.1(a, b)]{Ha97}. If $C^2<0$,
then $0>C^2=2p_a(C)-2-K_X\cdot C>2p_a(C)-2$ by adjunction, hence
$p_a(C)=0$ and $C$ is smooth and rational with $C^2=-1$.
Consider
$$0\to \O_X\to \O_X(C)\to\O_C(C)\to 0.\eqno(\star)$$
Since $X$ is rational we have $h^2(X,\O_X)=h^1(X,\O_X)=0$. Since $C$ is smooth and rational with
$C^2=-1$, we have $h^1(C, \O_C(C))=0$. Thus $h^1(X,\O_X(C))=0$.

We are left with the case that $D\cdot C=-K_X\cdot C=0$, hence $C$ is disjoint from $D$.
Thus $\O_C(C)=K_C$ by adjunction, so
$h^1(C, \O_C(C))=1$, and therefore from $(\star)$ we have $h^1(X, C)=1$.
\end{proof}

\subsection{Speciality on blow-ups of $\P^2$}

   In this section we look at  the speciality of nef linear systems on blow-ups of $\P^2$.
   We try to formulate a statement, which is valid without assuming that the points
   we blow up are in general position. The examples presented here suggest that
   one needs to reformulate the problem stated in Section \ref{specstab}.

   The first example shows that multiples of a special effective nef linear system might
  no longer be special.

\begin{example}\rm
   This example is based on the existence of very ample but special linear systems.
   Let $p_1,\dots,p_{25}$ be transversal intersection points of two
   smooth curves of degree $5$. Let  $f:X\to \P^2$ be the blow-up
   of the plane at these 25 points. By \cite[p.  796]{GGH94}, for
   $1\leq r\leq 2$ and any $m>0$,
   the linear system
   $$L = (5m+r)H-m(E_1+\cdots+E_{25}),$$
   where $H$ is the class of the line and $E_i$ are the exceptional divisors of $f$,
   is very ample and special, but Serre vanishing
   implies that some multiple of $L$ is no longer special. Let $X'\to X$ be the blow-up of $X$
   at an additional point, and let $L'$ be the pullback to $X'$ of $L$.  Then we have an example
   of a linear system $L'$ which is special and nef but not ample and
   for which $sL'$ is non-special for $s\gg0$.
\end{example}

   The second example shows that speciality might persist.
\begin{example}\rm
   Let $C$ and $D$ be smooth plane curves of degree $d\geq 3$ intersecting transversally
   in $p_1,\dots,p_{d^2}$. Let $f:X\to\P^2$ be the blow-up of the plane
   at these points. The linear system
   $$L=dH-\sum_{i=1}^{d^2}E_i$$
   is special (again because its virtual dimension is negative).
   The same remains true for all multiples of $L$. This follows
   from the restriction sequence
   $$0\to mL\to (m+1)L\to (m+1)L|_C\to 0.$$
   Indeed, by Serre duality we have $h^2(mL)=0$ for all $m\geq 1$.
   Also, as $L$ is a pencil of disjoint curves, we have $L|_C=\OO_C$.
   Then taking the long cohomology sequence of the restriction sequence we have the mapping
   $$\cdots\to H^1(X,(m+1)L)\to H^1(C,\OO_C)\to 0.$$
   The assumption $d\geq 3$ guarantees that $C$ is non-rational, hence $h^1(C,\OO_C)=g(C)>0$.
\end{example}

   The next example shows that speciality may increase linearly while
   the number of global sections remains fixed.
\begin{example}\rm
   Let $C\subset\P^2$ be a quartic curve with a simple node at $p_0$ and smooth otherwise.
   Such a curve exists
   by a simple dimension count.
   Then we can take $12$ points $p_1,\dots,p_{12}$ on $C$ in such a way that $C$ is the
   unique quartic passing through $p_0$ with multiplicity $2$ and through $p_1,\dots,p_{12}$
   with multiplicities all equal to $1$. Let $f:X\to\P^2$ be the blow-up of the plane at
   the points $p_0,\dots,p_{12}$. We consider the linear system
   $L=4H-2E_0-E_1-\dots-E_{12}$ on $X$. By a slight abuse of notation we write
   $C$ also for the proper transform of $C$ on $X$. It is a smooth curve
   of genus $2$. If the points $p_1,\dots,p_{12}$ are generic
   enough, then $mL|_C$ has no global sections for all $m\geq 1$
   and by Riemann-Roch on $C$ we have $h^1(C,mL|_C)=1$. Using again the restriction sequence
   we get
   $$h^0(X,mL)=1\;\;\mbox{ and }\; h^1(X,mL)=m\;\;\mbox{ for all }\; m\geq 1.$$
\end{example}
   These examples suggest the following problem.
\begin{problem}
   Let $X$ be the blow-up of $\P^2$ at $r$ distinct points ($r$ is arbitrary and
   also the position of the points is arbitrary but we require the points
   to be distinct) and let $L$ be an effective nef divisor on $X$.
   Then is it true that either
   \begin{itemize}
      \item[a)] there exists $m\geq 1$ such that $h^1(mL)=0$ (and this $m$ should
      be $1$ if the points are in general position); or
      \item[b)] there is a (not
necessarily irreducible) curve $C$ on $X$ such that $|L-C|\ne \emptyset$,
$p_a(C)>0$ and $L\cdot C=0$?
   \end{itemize}
\end{problem}

\subsection{Bounded negativity}

   In the course of the discussions, we investigated the following problem,
   see \cite[Conjecture 1.2.1]{Har10}.

\begin{conjecture}[Bounded Negativity]\label{BNC}
   Let $X$ be a smooth projective surface in
   characteristic zero. There exists a positive constant $b(X)$
   bounding the self-intersection of reduced, irreducible curves on $X$, i.e.,
   $$
      C^2 \ge -b(X)
   $$
   for every reduced, irreducible curve $C\subset X$.
\end{conjecture}

   The restriction to characteristic zero is in general essential (see Example \ref{charpexample})
   and of course so is the hypothesis that the curves be
   reduced, but it is not necessarily essential that they be irreducible. See Section \ref{RedBN}
   for further discussion.

   One situation where Bounded Negativity holds is when the
   anti-canonical divisor is $\mathbb Q$-effective
   (see \cite[I.2.3]{Har10}):

\begin{proposition}\label{prop-anti-eff-BNC}
   Let $X$ be a smooth projective surface such that for some
   integer $m>0$ the pluri-anti-canonical divisor $-mK_X$ is
   effective.
   Then there exists a
   positive constant $b(X)$ such that
   $C^2 \geq -b(X)$ for every irreducible curve $C$ on $X$.
\end{proposition}

\begin{proof}
   As $-mK_X$ is effective, there exist only finitely
   many irreducible curves $C$ such that $-K_X\cdot C <0$.  Hence
   apart from these finitely many prime divisors, we have $-K_X \cdot C
   \geq 0$, in which case by the adjunction formula
   \[ C^2 = 2p_a-2 -K_X \cdot C \geq 2p_a-2 \geq -2.\]
\end{proof}
   Note that the hypothesis of $\Q$-effectivity holds for instance
   on toric surfaces.
   Another case where the conjecture is clearly true is when
   $-K_X$ is nef (with the same argument as that in the proof of
   Proposition~\ref{prop-anti-eff-BNC}).

   The following example (see also \cite[Remark 1.2.2]{Har10}) shows that
   the restriction on the characteristic in Conjecture \ref{BNC}
   cannot be avoided in general (but perhaps it can be avoided by restricting
   $X$ to be, for example, rational).

\begin{example}[Exercise V.1.10, \cite{refHr77}]\label{charpexample}\rm
   Let $C$ be a smooth curve of genus $g\geq 2$ defined over a field
   of characteristic $p>0$ and let $X$ be the product surface $X=C\times C$.
   The graph $\Gamma_q$ of the Frobenius morphism defined by taking
   $q=p^r$--th powers is a smooth
   curve of genus $g$ and self-intersection $\Gamma_q^2=q(2-2g)$.
   With $r$ going to infinity, we obtain a sequence of smooth curves
   of fixed genus with self-intersection going to minus infinity.
\end{example}

   In view of this example it is interesting to ask if at least one
   of the following is true.

\begin{conjecture}[Weak Bounded Negativity]\label{WBNC}
   Let $X$ be a smooth projective surface in
   characteristic zero and let $g\geq 0$ be an integer.
   There exists a positive constant $b(X,g)$
   bounding the self-intersection of curves of geometric genus $g$ on $X$, i.e.,
   $$
      C^2 \ge -b(X,g)
   $$
   for every reduced, irreducible curve $C\subset X$ of geometric genus $g$ (i.e., the genus of
   the normalization of $C$).
\end{conjecture}

\begin{conjecture}[Very Weak Bounded Negativity]\label{VWNBC}
   Let $X$ be a smooth projective surface in
   characteristic zero and let $g\geq 0$ be an integer.
   There exists a positive constant $b_s(X,g)$
   bounding the self-intersection of smooth curves of geometric genus $g$ on $X$, i.e.,
   $$
      C^2 \ge -b_s(X,g)
   $$
   for every irreducible smooth curve $C\subset X$ of geometric genus $g$.
\end{conjecture}

   It would be interesting to know if a bound of this type extends
   to a family of surfaces, i.e., if the following question has an affirmative answer.

\begin{problem}
   Let $f:Y\to B$ be a morphism from a smooth projective threefold~$Y$
   to a smooth curve $B$ such that the general fibre is a smooth surface.
   Is there a constant $b(Y,g)$ such that
   $$
      C^2 \geq -b(Y,g)
   $$
   for all vertical curves $C\subset Y$ (i.e., $f(C)=\mbox{a point}$) of geometric genus $g$?
   (Here the self-intersection is computed in the fibre of $f$ containing $C$.)
\end{problem}

   We show that to a large extent both Conjectures \ref{WBNC} and \ref{VWNBC}
   are true. The key ingredient of the proofs is the following vanishing result,
   \cite[Corollary 6.9]{EV}. We recall the basic properties of log
   differential forms in Appendix \ref{appendixA}.

\begin{theorem}[Bogomolov-Sommese Vanishing]\label{thm:BSV}
   Let $X$ be a smooth projective variety defined over an algebraically
   closed field of characteristic zero, $L$ a line bundle,
   $C$ a normal crossing divisor on $X$. Then
   $$
     H^0(X, \Omega^{a}_X(\log C)\otimes L^{-1}) = 0
   $$
   for all $a<\kappa(X,L)$.
\end{theorem}

\begin{corollary}\label{cor:notbig}
   Let $X$ be a smooth projective surface defined over an algebraically
   closed field of characteristic zero, $C$ a normal crossing divisor on $X$.
   Then $\olc$ contains no big line bundles.
\end{corollary}
\begin{proof}
   Note that $L$ being big means $\kappa(X,L)=2$.
   An inclusion $L\hookrightarrow\olc$ gives rise to a non-trivial section of
   $H^0(X,\olc\otimes L^{-1})$, which vanishes according to Theorem~\ref{thm:BSV}.
\end{proof}

\begin{remark}\rm
   This kind of result was first observed by Bogomolov for the cotangent
   bundle of a surface itself:
   \begin{equation}\label{thm:BV}
      \mbox{If }\; L\subset\Omega_X^1 \mbox{ is a line bundle, then } L \mbox{ is not big.}
   \end{equation}
   We refer to \cite[Proposition 2.2]{VdV}
   for a nice and detailed proof.
\end{remark}

\begin{remark}\rm
   Theorem \ref{thm:BSV} and statement \eqnref{thm:BV} are known to be false in positive
   characteristic, see \cite[Remark 7.1]{Kol95}.
\end{remark}

\subsection{Partial proof of Conjecture \ref{VWNBC}: Very Weak Bounded Negativity}

Here we prove Conjecture \ref{VWNBC} on surfaces with
   Kodaira dimension $\kappa(X)\geq 0$.

   Let $X$ be a surface as in Theorem \ref{thm:BSV} and let $\shf$ be a coherent sheaf of rank $r$ on $X$.
   Recall that the \emph{discriminant} $\Delta(\shf)$ is defined as
   $$
      \Delta(\shf) := 2r c_2(\shf) - (r-1)c_1(\shf)^2\ .
   $$
   If $\shf$ is a rank $2$ vector bundle, then this reduces to
   $$
     \Delta(\shf) \equ  4c_2(\shf) - c_1(\shf)^2.
   $$
   The interest in the discriminant of a vector bundle stems in part
   from the following useful numerical criterion for stability of
   vector bundles on surfaces. Recall first
\begin{definition}\label{def:bstable}\rm
   Let $\shf$ be a rank $2$ vector bundle on a smooth projective surface $X$.
   We call $\shf$ \emph{unstable} if there exist line bundles $A$ and $B$
   on $X$ and a finite subscheme $Z\subset X$ (possibly empty) such that
   the sequence
   $$0\to A\to \shf\to B\otimes\cali_Z\to 0$$
   is exact and the $\Q$-divisor
   $$P:=A-\frac{1}{2}c_1(\shf)$$
   satisfies the conditions
   $$P^2>0 \;\;\mbox{ and }\;\; P\cdot H>0$$
   for all ample divisors $H$ on $X$.
\end{definition}
\begin{remark}\rm
   The conditions in the definition imply that $P$ is a big divisor.
\end{remark}
   The fundamental result of Bogomolov \cite{Bog77} is the following numerical criterion
\begin{theorem}[Bogomolov]
   Let $\shf$ be a rank $2$ vector bundle on a smooth projective surface $X$.
   If $\Delta(\shf)<0$, then $\shf$ is unstable.
\end{theorem}
   As a corollary we prove Very Weak Bounded Negativity for smooth curves
   on surfaces with $\kappa(X)\geq 0$.
\begin{proposition}\label{prop:VWBNC}
   Let $X$ be a smooth projective surface with $\kappa(X)\geq 0$ and let
   $C\subset X$ be a smooth curve of genus $g(C)$. Then
   $$C^2\geq c_1^2(X)-4c_2(X)-4g(C)+4.$$
\end{proposition}
\begin{proof}
   We consider the following exact sequence which arises via an elementary
   transformation of $\Omega^1_X$ along $\Omega^1_C$ (see \cite[Example 5.2.3]{HuyLeh})
   $$0\to W:=\olc\otimes\cali_C\to \Omega^1_X\to \Omega^1_C\to 0.$$
   By \cite[Proposition 5.2.2]{HuyLeh} we have
   $$c_1(W)=(K_X+C)-2C=K_X-C\;\;\mbox{ and }\;\; c_2(W)=c_2(\Omega^1_X)+\deg(\Omega^1_C)-K_X\cdot C.$$
   Hence
   $$\Delta(W)=4c_2(X)-c_1^2(X)+4g-4+C^2.$$
   If we assume that $C^2< c_1^2(X)-4c_2(X)-4g(C)+4$, then $\Delta(W)<0$ and $W$ is unstable.
   According to Definition \ref{def:bstable} there exists then a line bundle $A\subset W$
   such that
   $$A-\frac12c_1(W)=A+\frac12(C-K_X) \;\;\mbox{is big.}$$
   It follows that
   $$A+C\subset \olc$$
   and since $\kappa(X)\geq 0$ and $C$ is effective
   $$A+C=\left(A+\frac12C-\frac12K_X\right)+\frac12K_X+\frac12C$$
   is big as well. However, this contradicts the Bogomolov-Sommese Vanishing, Theorem \ref{thm:BSV}.
\end{proof}
\begin{remark}\rm
   Note that a statement as in the above Proposition cannot hold
   on ruled surfaces, i.e., there is no lower bound on $C^2$ with
   $C$ smooth depending only on the genus of $C$ and Chern numbers
   of $X$. Indeed, all Hirzebruch surfaces
   $F_n=\P(\calo_{\P^1}\oplus \calo_{\P^1}(-n))$
   have the same invariants $c_1^2(F_n)=8$, $c_2(F_n)=4$
   and each $F_n$ contains a smooth rational curve
   of self-intersection $-n$.
      This observation is of course not a counterexample
   either to Conjecture \ref{BNC} or to Conjecture \ref{VWNBC}.
   It merely means that the bound for $C^2$ on ruled surfaces
   must depend on something else.
   A reasonable possibility, in the case where $X$ is obtained by blowing up $r$ points
   of a surface $Y$ when $b(Y)$ exists, is to try to show that $b(X)$ can be defined in terms
   of $b(Y)$ and $r$.
\end{remark}

\subsection{Partial proof of Conjecture \ref{WBNC}: Weak Bounded Negativity}

   The main auxiliary ingredient in this part is the logarithmic version
   of the Miyaoka-Yau inequality.
\begin{theorem}[Logarithmic Miyaoka-Yau inequality]\label{thm:lmyi}
   Let $X$ be a smooth projective surface and $C$ a smooth curve
   on $X$ such that the adjoint line bundle $K_X+C$ is $\Q$--effective,
   i.e., there is an integer $m>0$ such that $h^0(m(K_X+C))>0$. Then
   $$c_1^2(\olc)\leq 3c_2(\olc),$$
   equivalently $(K_X+C)^2\leq 3\left(c_2(X)-2+2g(C)\right)$.
\end{theorem}
   We refer to Appendix~\ref{appendixA} for the proof. Note
   that our statement of this inequality is not the most general,
   but it suffices for our needs.  We also need the following elementary lemma.
\begin{lemma}\label{lem:blowinv}
   Let $X$ be a smooth projective surface, $C\subset X$ a reduced, irreducible
   curve of geometric genus $g(C)$, $P\in C$ a point with $\mult_PC\geq 2$.
   Let $\sigma:\tilde{X}\to X$ be the blow-up of $X$ at $P$ with the
   exceptional divisor $E$. Let $\tilde{C}=\sigma^*(C)-mE$ be the proper
   transform of $C$. Then the inequality
   $$\tilde{C}^2\geq c_1^2(\tilde{X})-3c_2(\tilde{X})+2-2g(\tilde{C})$$
   implies
   $$C^2\geq c_1^2(X)-3c_2(X)+2-2g(C).$$
\end{lemma}
\begin{proof}
   We have
   $$C^2=\tilde{C}^2+m^2,\;\; c_1^2(X)=c_1^2(\tilde{X})+1,\;\; c_2(X)=c_2(\tilde{X})-1\;
   \mbox{ and }\; g(C)=g(\tilde{C}).$$
   Hence
   $$\begin{array}{ccl}
     C^2 & = & m^2 +\tilde{C}^2 \\
     	& \geq &m^2 + c_1^2(\tilde{X})-3c_2(\tilde{X})+2-2g(\tilde{C})\\
         & = & m^2 + c_1^2(X)-1 - 3c_2(X)-3 +2-2g(C)\\
         & \geq & c_1^2(X)-3c_2(X)+2-2g(C).
     \end{array}$$
\end{proof}
\begin{proposition}\label{prop:WBNC}
   Let $X$ be a smooth projective surface with $\kappa(X)\geq 0$.
   Then for every reduced, irreducible curve $C\subset X$ of geometric genus $g(C)$ we have
   \begin{equation}\label{eq:weakbound}
      C^2\geq c_1^2(X)-3c_2(X)+2-2g(C).
   \end{equation}
\end{proposition}
\begin{proof}
   The idea is to reduce the statement to a smooth curve and use
   Theorem \ref{thm:lmyi}.

   We blow up $f:\tilde{X}=X_N\to X_{N-1}\to\ldots\to X_0=X$ resolving
   step-by-step the singularities of $C$. The proper transform of $C$ in $\tilde{X}$
   is then a smooth irreducible curve $\tilde{C}$. Applying Lemma \ref{lem:blowinv}
   recursively to every step in the resolution $f$ we see that it is enough
   to prove inequality \eqnref{eq:weakbound} for $C$ smooth.

   This follows easily from the Logarithmic Miyaoka-Yau inequality \ref{thm:lmyi}.
   Note that our assumption $\kappa(X)\geq 0$ implies that $K_{\tilde{X}}+\tilde{C}$
   is $\Q$--effective. Hence
   $$c_1^2(X)+2C\cdot(K_X+C)-C^2=c_1^2(\olc)\leq 3c_2(\olc)=3c_2(X)-6+6g(C).$$
   By adjunction $C\cdot(K_X+C)=2g(C)-2$ and rearranging terms we arrive at \eqnref{eq:weakbound}.
\end{proof}
\begin{remark}\rm
   Note that Proposition \ref{prop:WBNC} applies in particular to smooth curves
   and provides in general a better bound than that in Proposition \ref{prop:VWBNC}.
   We do not pursue the optimality problem here, and we found it instructive
   to provide two possible proofs for the Very Weak Bounded Negativity Conjecture.
\end{remark}

\subsection{Bounded negativity conjecture and Seshadri constants}

   We next point out an
   interesting connection between bounded negativity and a question on Seshadri
   constants posed by
   Demailly in \cite[Question 6.9]{Dem92}:

   \begin{problem}Is the global Seshadri constant
   $$
      \eps(X):=\inf\set{\eps(L)\with L\in\Pic(X)\mbox{ ample}}
   $$
   positive for every smooth projective surface $X$? \end{problem}
   At present,  this is unknown. In fact, it is  unknown whether
   for every fixed $x\in X$
   the
   quantity
   $$
      \eps(X,x)=\inf\set{\eps(L,x)\with L\in\Pic(X)\mbox{ ample}}
   $$
   is always positive. The latter, however, would be a
   consequence of the Bounded Negativity Conjecture:

\begin{proposition}
   If the Bounded Negativity Conjecture is true, then
   $$
      \eps(X,x)>0
   $$
   for every smooth projective surface $X$ in characteristic zero,
   and every $x\in X$.
\end{proposition}

   The proof below actually
   gives an effective lower bound on $\eps(X,x)$:
   If $Y=\Bl_x(X)$ is the
   blow-up of $X$ at $x$, then
   $$
      \eps(X,x)\ge\frac1{\sqrt{b(Y)+1}}
   $$
   So if  we knew that the constant
   $b(Y)$ is the same for every one-point blow-up of $X$ (or
   at least bounded from below by a constant that is
   independent of $x$), then we
   would get a lower bound on $\eps(X)$.

\begin{proof}[Proof of the proposition]
   Let $C\subset X$ be an irreducible curve of multiplicity $m$
   at $x$, and let
   $\tilde C\subset Y$ be its proper transform on the blow-up
   $Y$ of $X$ in $x$.
   Then
   $$
      C^2-m^2=(f^*C-mE)^2=\tilde C^2\ge -b(Y) \ .
   $$
   Consider first the case where $m\le\sqrt{b(Y)}$. Then
   $$
      \frac{L\cdot C}{m}
      \ge\frac{L\cdot C}{\sqrt{b(Y)}}
      \ge\frac1{\sqrt{b(Y)}}
   $$
   In the alternative case, where $m>\sqrt{b(Y)}$, we have
   $$
      C^2\ge m^2-b(Y)>0
   $$
   and hence, using the Index Theorem, we get
   $$
      \frac{L\cdot C}m
      \ge\frac{\sqrt{L^2}\sqrt{C^2}}m
      \ge\sqrt{1-\frac{b(Y)}{m^2}}
      \ge\sqrt{1-\frac{b(Y)}{b(Y)+1}}=\frac1{\sqrt{b(Y)+1}}.
   $$
\end{proof}

   An alternative argument for the proof of the proposition goes
   as follows. Suppose that $\eps(X,x)=0$. Then there is a
   sequence of ample line bundles $L_n$ and curves $C_n$ such
   that
   $$
      \frac{L_n\cdot C_n}{m_n}\to 0 \ ,
   $$
   where $m_n$ denotes the multiplicity of $C_n$ at $x$. We use
   now that the blow-up $Y=\Bl_x(X)$ has bounded negativity, so
   that for the proper transform $\tilde C_n$ of $C_n$ we have
   $$
      (\tilde C_n)^2\ge -b(Y)^2 \ .
   $$
   But $(\tilde C_n)^2=(f^*C_n-m_nE)^2=C_n^2-m_n^2$, which, upon using
   the Index Theorem, tells us that
   $$
      m_n^2-b(Y) \le C_n^2 \le \frac{(L_n\cdot C_n)^2}{L_n^2} \ .
   $$
   So we see that
   $$
      1-\frac{b(Y)}{m_n^2}\le\frac{C_n^2}{m_n^2}\le\frac{(L_n\cdot C_n)^2}{m_n^2}\cdot\frac1{L_n^2}\le\left(\frac{L_n\cdot C_n}{m_n}\right)^2 \ .
   $$
   Now, the left hand side of this chain of inequalities tends to
   one, whereas the right hand side goes to zero, a
   contradiction.

\subsection{The Weighted Bounded Negativity Conjecture}  

The following conjecture is yet another variant of the Bounded Negativity Conjecture \ref{BNC}:

\begin{conjecture}[Weighted Bounded Negativity]\label{WeightedBNC}
   Let $X$ be a smooth projective surface in
   characteristic zero. There exists a positive constant $b_w(X)$
   such that
      $$
      C^2 \ge -b_w(X)(H \cdot C)^2
   $$
   for every irreducible curve $C\subset X$ and every big and nef line bundle $H$ satisfying $H \cdot C >0$.
\end{conjecture}

The interest in this conjecture is due to the fact that this conjecture is sufficient to imply the conclusion of the previous proposition, by the following result.

\begin{proposition}
   If the Weighted Bounded Negativity Conjecture is true, then
   $$
      \eps(X,x)>0
   $$
   for every smooth projective surface $X$ in characteristic zero,
   and every $x\in X$.
\end{proposition}

\begin{proof}
   As above, there are two proofs, and one of these yields the effective lower bound
$$
      \eps(X,x)\ge\frac1{\sqrt{b_w(Y)+1}},
   $$
where $Y=\Bl_x(X)$ is the blow-up of $X$ at $x$. We first look at this proof.

Let $C\subset X$ be an irreducible curve of multiplicity $m$
   at $x$ and $L$ an ample line bundle on $X$. Let
   $\tilde C\subset Y$ be the proper transform of $C$ on $Y$.
 Since $f^*L$ is big and nef and $f^*L \cdot \tilde C = L \cdot C >0$, the Weighted Bounded Negativity Conjecture on $Y$ yields the existence of a constant $b_w(Y)$ such that
   $$
      C^2-m^2 = (f^*C-mE)^2= (\tilde C)^2 \ge -b_w(Y) (f^*L \cdot \tilde C)^2 = b_w(Y) (L \cdot C)^2.
   $$
Then, using the Index Theorem, we obtain
$$
  \frac{(L \cdot C)^2}{m^2}  \ge  \frac{L^2 C^2}{m^2} \geq L^2
   \Big(1-b_w(Y)\frac{(L \cdot C)^2}{m^2}\Big),
$$
that is,
$$
\Big(\frac{L \cdot C}{m}\Big)^2 \Big(1+b_w(Y)L^2\Big) \geq L^2.
$$
This yields
$$
\frac{L \cdot C}{m} \geq \sqrt{\frac{L^2}{1+b_w(Y)L^2}}= \sqrt{\frac{1}{\frac{1}{L^2}+b_w(Y)}}
 \geq \sqrt{\frac{1}{1+b_w(Y)}} = \frac{1}{\sqrt{1+b_w(Y)}},
$$
as asserted.

The second version of the proof goes as follows:  Suppose that $\eps(X,x)=0$. Then there is a
   sequence of ample line bundles $L_n$ and curves $C_n$ such
   that
   $$
      \frac{L_n\cdot C_n}{m_n}\to 0
   $$
   where $m_n$ denotes the multiplicity of $C_n$ at $x$.
   Let $f:Y=\Bl_x(X) \rightarrow X$ be the blow-up of $X$ at $x$ and
   denote by $\tilde C_n$ the proper transform of $C_n$. Since $f^*L_n$ is big and nef and $f^*L_n \cdot \tilde C_n = L_n \cdot C_n >0$, the Weighted Bounded Negativity Conjecture on $Y$ yields the existence of a constant $b_w(Y)$ such that
   $$
      (\tilde C_n)^2 \ge -b_w(Y) (f^*L_n \cdot \tilde C_n)^2 = -b_w(Y) (L_n \cdot C_n)^2.
   $$
   But $(\tilde C_n)^2=(f^*C_n-m_nE)^2=C_n^2-m_n^2$, which yields
   $$
   \frac{C_n^2}{m_n^2}-1 \ge -b_w(Y) \frac{(L_n \cdot C_n)^2}{m_n^2}.
   $$
   Combining with the Index Theorem, we obtain
      $$
      1-b_w(Y) \Big(\frac{L_n \cdot C_n}{m_n}\Big)^2 \leq \frac{C_n^2}{m_n^2}
        \leq \Big(\frac{L_n \cdot C_n}{m_n}\Big)^2 \cdot \frac{1}{L_n^2}.
   $$
   But the left hand side of this chain of inequalities tends to
   one, whereas the right hand side tends to zero, a
   contradiction.
\end{proof}

Note that the second proof does not give an effective lower bound for $\varepsilon(X,x)$.

\subsection{Bounded negativity for reducible curves}\label{RedBN}

We now consider a reducible version of Conjecture \ref{BNC}:

\begin{conjecture}[Reducible Bounded Negativity]\label{RBNC}
   Let $X$ be a smooth projective surface in
   characteristic zero. Then there exists a positive constant $b'(X)$
   bounding the self-intersection of reduced curves on $X$, i.e.,
   $$
      C^2 \ge -b'(X)
   $$
   for every reduced (but not necessarily irreducible) curve $C\subset X$.
\end{conjecture}

Conjecture \ref{BNC} implies Conjecture \ref{RBNC} in a very explicit way.

\begin{proposition}\label{RBNprop}
Let $X$ be a smooth projective surface (in any characteristic) for which
there is a constant $b(X)$ such that $C^2\geq -b(X)$ for every reduced, irreducible curve $C\subset X$.
Then $C^2\geq -\rho(X)b(X)\lceil b(X)/2\rceil$ for every reduced curve $C\subset X$,
where $\rho(X)$ is the Picard number of $X$ (i.e., the rank of the N\'eron-Severi group ${\rm NS}(X)$ of $X$).
\end{proposition}

\begin{proof}
Let $C=C_1+\cdots+C_r$ be a sum of distinct curves $C_i$ on $X$.
By reindexing we may write this sum as $C=\sum_{ij}C_{ij}$,
where for each $i$, the sets $\{C_{ij}\}_j$ are linearly independent in the
N\'eron-Severi group, and for $i<i'$, the span of $\{C_{ij}\}_j$ contains the span of $\{C_{i'j}\}_j$.

For each $i$, let $B_i=\sum_jC_{ij}$ and let $\beta$ be the number of elements $B_i$.
Then, $B_i^2\geq \sum_jC_{ij}^2\geq-\rho(X)b(X)$, where
the first inequality is because all of the cross terms are non-negative, and the second because
$C_{ij}^2\geq-b(X)$ but there can be at most $\rho(X)$ linearly independent elements in the N\'eron-Severi group.

Also note for any distinct prime divisors $A_1,\ldots,A_n, A$ such that the $A_i$ are linearly independent in
${\rm NS}(X)$ and $A^2<0$ with $A$
in the span (in ${\rm NS}(X)$) of the $A_i$, we have $A\cdot A_i>0$ for some $i$, and hence
$(A_1+\cdots+A_n)\cdot A>0$. This is because
we can write $A+\sum_sa_sA_s=\sum_ta_tA_t$ for non-negative rational coefficients $a_s$ and $a_t$,
with the sums over $s$ and $t$ running over disjoint subsets of $\{1,\ldots,n\}$. Since $A\neq A_i$ for all $i$,
we have $A\cdot A_i\geq0$ for all $i$, but if $A\cdot A_i=0$ for all $i$, then we would have
$0>A^2=A\cdot(A+\sum_sa_sA_s)=A\cdot(\sum_ta_tA_t)=0$.
In particular this shows for $i<i'$ that $B_i\cdot C_{i'j}>0$.

Let $u=\min(\beta,\lceil b(X)/2\rceil)$. We will now show that $C^2\geq (B_1+\cdots+B_u)^2$.
The result follows from this, since $(B_1+\cdots+B_u)^2\geq B_1^2+\cdots+B_u^2\geq -u\,\rho(X)b(X)
\geq-\rho(X)b(X)\lceil b(X)/2\rceil$.
If $u=\beta$, then $C=B_1+\cdots+B_u$, and clearly $C^2\geq (B_1+\cdots+B_u)^2$.
Otherwise, $C=B_1+\cdots+B_u+D_1+\cdots+D_t$, where $D_1,\ldots,D_t$ are the terms of
the sum $C=C_1+\cdots+C_r$ not already subsumed by $B_1+\cdots+B_u$.
But $C^2\geq (B_1+\cdots+B_u)^2 + \sum_i2(B_1+\cdots+B_u)\cdot D_i+\sum_iD_i^2$,
and $B_i\cdot D_j>0$ for each $i$ and $j$ (by our observation above regarding $B_i\cdot C_{i'j}>0$), so
$\sum_i2(B_1+\cdots+B_u)\cdot D_i+\sum_iD_i^2\geq \sum_i(2u+D_i^2)\geq \sum_i(2u-b(X))\geq0$.
Thus $C^2\geq (B_1+\cdots+B_u)^2$, as claimed.
\end{proof}

It seems unlikely that the bound given in Proposition \ref{RBNprop} is ever sharp.
For a blow-up $X$ of $\P^2$ at $n$ generic points, we expect by the SHGH Conjecture
that $C^2\geq -1$ for any reduced, irreducible curve $C\subset X$.
The bound given in Proposition \ref{RBNprop} would then be $C^2\geq -(n+1)$ for
reduced but possibly reducible curves $C$.
It is easy to produce reduced examples with $C^2=-n$; for example, take $C=E_1+\cdots+E_n$.
   On the other hand, one cannot reach $C^2=-(n+1)$. Indeed, the only
   possibility would be to have $n+1$ disjoint $(-1)$ curves. As the Picard
   number of the blow-up is $(n+1)$, this would contradict the Index Theorem.

Thus it is of interest to give specific examples of reduced curves with $C^2$ as negative as possible.
Clearly the negativity of $C^2$ can grow with $\rho(X)$, so it makes sense to normalize $C^2$.
We will consider several examples involving blow-ups of $\P^2$ at $n$ points,
with the goal of finding values of $C^2/n$ that are as negative as possible.

\begin{example}\rm
Let $X$ be the blow-up of $\P^2$ at $n$ collinear points, and let $C$ be the proper transform of the line
containing the points. Then $C^2/n$ is approximately $-1$. Here $C$ is a prime divisor.
\end{example}

\begin{example}\rm
Consider a general map $f:\P^1\to \P^2$ of degree $d$. Let $C\subset X$ be the proper transform
of the image of $\P^1$ (which has degree $d$) where $X$ is obtained by blowing up the singular points
of $f(\P^1)$ (there are $n=\binom{d-1}{2}$ nodes). Thus $C^2/n$ is approximately $-2$.
Here $C$ is again a prime divisor.
\end{example}

\begin{example}\rm
Let $X$ be the blow-up of the points of intersection of $s>2$ general lines in $\P^2$;
thus $n=\binom{s}{2}$. Let $C$ be the proper transform of the union of the lines.
Then $C^2=s(s-2)$, so $C^2/n$ is approximately $-2$.
\end{example}

\begin{example}\rm
Assume the ground field $k$ is algebraically closed of characteristic $p$. Let $q$ be some power of $p>0$.
Blow up the points of $\P^2$ with coordinates in the finite field ${\mathbb F}_q$. Let $C$ be the proper transform
of the union of all lines through pairs of the points blown up. There are $n=q^2+q+1$ points, and also
$q^2+q+1$ lines, and $q+1$ lines pass through each of the points. Thus
$C^2=(q^2+q+1)^2-(q+1)^2(q^2+q+1)=-q(q^2+q+1)$, so $C^2/n=-q$, which is approximately $-\sqrt{n}$.
\end{example}

\begin{remark}\rm
This last example raises the question of how negative $C^2$ can be if $C$ is the proper transform
of a union of lines. This makes contact with interesting problems studied by combinatorists.
For example, let $S\subset \P^2$ be a set of $n$ distinct points, and let $\Lambda=\{L_1,\ldots,L_l\}$
be a set of $l$ distinct lines.
Let $M(S,\Lambda)$ be the incidence matrix; its rows correspond to the points, its columns correspond to the
lines, and an entry is either 1 or 0 according to whether the corresponding point lies or does not lie
on the corresponding line. Let $|M|$ be the sum of the entries of $M$.

If $X$ is obtained by blowing up the points of $S$, and $C=L_1'+\cdots+L_l'$, where $L_i'$
is the proper transform of $L_i$, then $C^2\geq \sum_i(L_i')^2=l-|M|$, with equality
if $S$ contains all of the points where any two of the lines meet.

It's easy to get a coarse bound on $C^2/n$. Clearly, $l\leq \binom{n}{2}$, and $|M|\leq ln$,
so $C^2/n\geq \binom{n}{2}(1-n)/n\geq -\binom{n}{2}$.
For points with real coefficients and lines defined over the reals,
the Szemer\'edi-Trotter Theorem \cite{ST83} gives an order of magnitude estimate
$|M|\leq O((ln)^{2/3}+l+n)$. There also are Szemer\'edi-Trotter type results over finite fields; for example, see
Bourgain-Katz-Tao \cite{BNT04} and Vinh \cite{V07}.
\end{remark}

\subsection{Imposing higher vanishing order at one point}
   We next study the polynomial
   interpolation problem on the projective plane. We denote by ${\mathcal{L}}(d; m_1, m_2, \dots, m_n)$
   the linear system of homogeneous polynomials of degree $d$ passing
   through $n$ generic points with multiplicities at least $m_1,\dots, m_n$.
   We want to
   study ${\mathcal{L}}(d; m_1+1, m_2, \dots, m_n)$,
   provided ${\mathcal{L}}(d; m_1, m_2, \dots, m_n)\not= \emptyset$.

Let $\pi : X \to {\mathbb{P}}^2$ the blow-up of $n$ generic points; let $E_i=\pi^{-1}(p_i)$ be the exceptional divisors and $H=\pi^*L$. Consider the divisor $D=dH-\sum m_i E_i$.  We can identify
\[
{\mathcal{L}}(d; m_1, m_2, \dots, m_n) \equ H^0({\mathcal{O}}_X(D))
\]
and
\[
{\mathcal{L}}(d; m_1+1, m_2, \dots, m_n)=H^0({\mathcal{O}}_X(D-E_1)).
\]
Thus, the starting point for our problem  is to analyze the cohomology of the exact sequence
\[
0 \to {\mathcal{O}}_X(D-E_1) \to {\mathcal{O}}_X(D) \to {\mathcal{O}}_D(D_{\vert E_1}) \to 0.
\]
As a matter of fact, passing to the long exact cohomology sequence we have
\[
\begin{split}
0 & \to H^0({\mathcal{O}}(D-E_1)) \to H^0({\mathcal{O}}(D)) \xrightarrow{\rho} H^0({\mathcal{O}}(D_{\vert E_1})) \\
& \to H^1({\mathcal{O}}(D-E_1)) \to H^1({\mathcal{O}}(D))  \to 0
\end{split}
\]
The image of $\rho$ is a linear system on $E_1$ of degree $m_1$.

\begin{example}\rm
Consider $D=4H-E_1-2E_2-\cdots - 2E_5$,  then

\[\begin{array}{ccccccccc}
H^0(4;2^5) & \!\!\!\to\!\!\! & H^0(4; 1, 2^4) & \!\!\xrightarrow{\rho}\!\! & H^0({\mathcal{O}}_{{\mathbb{P}}{^1}}(1)) & \!\!\!\to\!\!\!& H^1(4;2^5) & \!\!\!\to\!\!\!& H^1(4; 1, 2^4)\\
\parallel & & \parallel & & \parallel & & \parallel  & & \parallel\\
1 & & 2 & & 2 & & 1 & & 0
\end{array}
\]

It is known that ${\mathcal{L}}(4;2^5)$ is a special system, since its expected dimension is $-1$ but there is an element in it: the quartic $2C$, where $C$ is the conic passing through the five points. Thus $ H^1(4;2^5)=1$, and in this case $\rho$ does  not have  maximal rank.
\end{example}

\begin{example}\rm
   As a slight modification of the above example we next consider
   $D=6H-2E_1-3E_2-\cdots - 3E_5$.  Then

\[\begin{array}{ccccccccc}
H^0(6;3^5) & \!\!\!\to\!\!\! & H^0(6; 2, 3^4) & \!\!\xrightarrow{\rho}\!\! & H^0({\mathcal{O}}_{{\mathbb{P}}{^1}}(2)) & \!\!\!\to\!\!\!& H^1(6;3^5) & \!\!\!\to\!\!\!& H^1(6; 2,3^4)\\
\parallel & & \parallel & & \parallel & & \parallel  & & \parallel \\
1 & & 2 & & 3 & & 3 & & 1
\end{array}
\]
   This shows that also in this case the rank of $\rho$ is not maximal.
\end{example}

The above examples give hints about the behavior of ${\mathcal{L}}(d; m_1+1, m_2, \dots, m_n)$, given ${\mathcal{L}}(d; m_1, m_2, \dots, m_n)\not= \emptyset$. We formulate the following:

\begin{conjecture}
Let $X$ be the blow-up of ${\mathbb{P}}^2$ in $r$ general points,  $D$ an effective divisor on $X$ and $E$ a $(-1)$-curve. For the restriction map $|D| \xrightarrow{\rho} |D_{\vert E}|$ one has:
\begin{enumerate}
\item All base points in the image of $\rho$ come from a base curve in $|D|$.
\item If $|D|$ has no fixed component, then $\rho$ has maximal rank.
\end{enumerate}
\end{conjecture}

If points in special position are allowed, then it is not hard to
find counterexamples to this statement. However, we would like to
consider the analogous question when $E$ is a $(-n)$-curve, with $n>1$.
We denote by
${\mathcal{L}}(d; \overline{m_1,  \dots, m_r}, m_{r+1}, \dots,  m_n)$
the linear system of curves of degree $d$ with points of multiplicities
$m_1, \dots, m_n$, the first $r$ of which are collinear, but which are
otherwise general.
   The discussions led to the following problem:

\begin{problem}
Is it true that, if $d \geq m_1 + \cdots + m_r$, then ${\mathcal{L}}(d; m_1,\dots,  m_n)\not= \emptyset$ if and only if  ${\mathcal{L}}(d; \overline{m_1,  \dots, m_r}, m_{r+1}, \dots,  m_n)\not= \emptyset$?
\end{problem}

This statement seems quite strong. If true, the induction arguments
common in approaches to the SHGH
conjecture as \cite{Hir85}, \cite{CM00}, \cite{Eva07} would be significantly
simplified, which might even lead to a proof.
However we want to remark that some genericity condition
is necessary here as well; in particular it is not possible to
allow the existence of many negative curves even if they meet the
system non-negatively. For example, if we pick four general lines, and four general points on each of them, there is a quartic, if not a pencil, through the $16$ points, which is not the case for general points.
It is easy to see that we can generalize this counterexample just considering  $d>3$ lines and $d$ points on each line. We don't have, at the moment, any counterexample to the statement with only two or three lines.

\subsection{Geometrization of Dumnicki's method \cite{MD}}\label{sec:geometrization of Dumnicki}
   Let  $D \subset {\mathbb{N}}^2$ be a finite set, such that $D=P \cap {\mathbb{N}}^2$, where $P$ is a convex polygon with integer vertices.
   We denote by $\shl(D)$ the linear series on the affine plane $\A^2$ spanned by monomials in $D$ (we identify
   a point $(k,l)\in\N^2$ with the monomial $x^ky^l$).
It can also be viewed as the complete linear
series associated to the polarized toric variety $(X_D,L_D)$ defined
by the polygon $D$, i.e., $\shl(D)=\mathbb{P}(|L_D|).$
We also write $\shl(D,m_1,\dots,m_r)_{p_1,\ldots,p_r}$ for
   the subsystem of $\shl(D)$ consisting of polynomials vanishing at the given $r$ smooth points with multiplicities
   at least $m_1,\dots,m_r$.
   One then defines:
   $$\shl(D,m_1,\dots,m_r)= \min_{p_1,\ldots,p_r\in X_D}\{\shl(D,m_1,\dots,m_r)_{p_1,\ldots,p_r}\}.$$
   Consider a partition $D=D_1 \cup D_2 \cup \dots \cup D_r$ determined by lines (not through the vertices).
   Then the following holds:
\vskip0.2cm
{\bf Fact 1:} If ${\mathcal{L}}(D_i,m_i)=\emptyset$, for all $i=1,\dots,r$, then ${\mathcal{L}}(D,m_1, \dots, m_r)=\emptyset$.
\vskip0.2cm
{\bf Fact 2:} Given $D \subset {\mathbb{N}}^2$  such that $|D|= \binom{m+1}{2}$, then  ${\mathcal{L}}(D,m)\ne\emptyset$ if and only if there exists
   a nonzero polynomial $F \in   {\mathbb{Q}}[x,y]$ with $\deg(F)=m-1$, such that $F(a,b)=0$ for all $(a,b) \in D$.
\vskip0.2cm
From Fact 2 (and B\'ezout's Theorem), we immediately obtain the following
\vskip0.2cm
{\bf Fact 3:} If, for $D \subset \mathbb{N}^2$, there exist $m$ horizontal (resp. vertical) lines $\ell_1,\dots,\ell_{m}$
such that
$$D \subset \bigcup_{k=1}^{m} \ell_k, \qquad \#(D \cap \ell_k) \leq k \text{ for } k=1,\dots,m$$
then $\mathcal{L}(D;m) = \emptyset$.
\vskip0.2cm

As an example, we prove that the divisor $13H-5E_1-4E_2- \dots -4E_{10}$ is not semi-effective (see Definition \ref{semieff} and Problem \ref{VEprob}).
For a fixed $n$, we consider the set $D = \{ (x,y) \in \mathbb{N}^2 : x+y \leq 13n \} \subset \mathbb{N}^2$.
We will cut $D$ with nine lines into ten subsets $D_1,\dots,D_{10}$. The equations of lines are given by the following functions (for small $\varepsilon > 0$):
$$
\begin{array}{rl}
f_1 : & -x-y+5n-1+\varepsilon,\\
f_2 : & x-9n-1+\varepsilon,\\
f_3 : & y-9n-1+\varepsilon,\\
f_4 : & x-y-5n-1+\varepsilon,\\
f_5 : & -x+y-5n-1+\varepsilon,\\
f_6 : & 3x-y-15n-1+\varepsilon,\\
f_7 : & -3x+y+3n-1+\varepsilon,\\
f_8 : & (2n+1)x-2ny-2n^2-5n-1+\varepsilon,\\
f_9 : & x+3y-21n+\varepsilon.
\end{array}
$$

\begin{figure}[ht!]
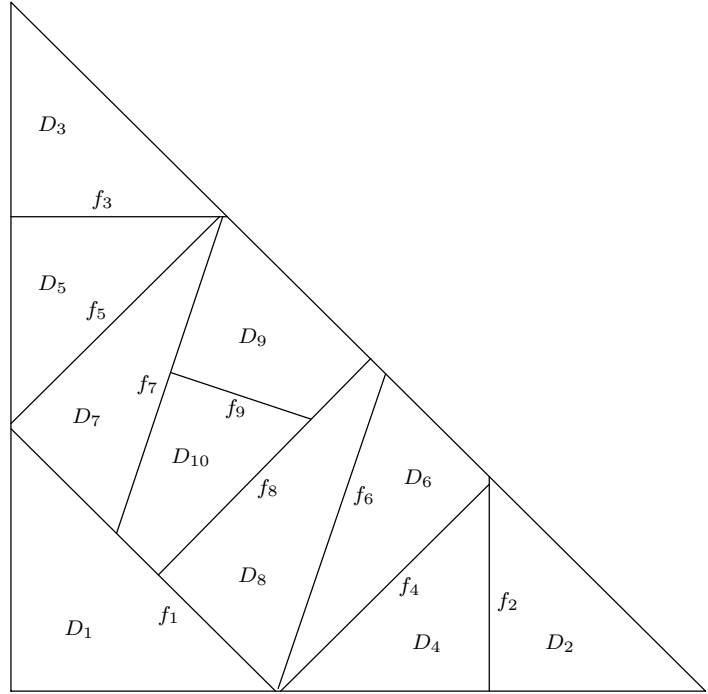

$$
\centertexdraw{
\drawdim pt
\linewd 0.5

\move(0 0)
\lvec(260 0)
\lvec(0 260)
\lvec(0 0)

\move(0 99)
\lvec(99 0)
\htext(55 25){{\scriptsize $f_1$}}

\move(179 0)
\lvec(179 81)
\htext(182 30){{\scriptsize $f_2$}}

\move(0 179)
\lvec(81 179)
\htext(30 182){{\scriptsize $f_3$}}

\move(101 0)
\lvec(179 78)
\htext(145 36){{\scriptsize $f_4$}}

\move(0 101)
\lvec(78 179)
\htext(28 140){{\scriptsize $f_5$}}

\move(100 1)
\lvec(140.25  119.75)
\htext(128 70){{\scriptsize $f_6$}}

\move(39.5 59.5)
\lvec(79.33 179)
\htext(47 112){{\scriptsize $f_7$}}

\move(55.07 43.92)
\lvec(134.58 125.46)
\htext(92 73){{\scriptsize $f_8$}}

\move(59.7 120.1)
\lvec(112.28 102.57)
\htext(80 103){{\scriptsize $f_9$}}

\htext(20 20){{\scriptsize $D_1$}}
\htext(200 15){{\scriptsize $D_2$}}
\htext(10 210){{\scriptsize $D_3$}}
\htext(150 15){{\scriptsize $D_4$}}
\htext(10 150){{\scriptsize $D_5$}}
\htext(147 77){{\scriptsize $D_6$}}
\htext(23 100){{\scriptsize $D_7$}}
\htext(85 40){{\scriptsize $D_8$}}
\htext(85 130){{\scriptsize $D_9$}}
\htext(60 85){{\scriptsize $D_{10}$}}

}
$$
\caption{Subdivision of $D$}
\end{figure}

The sets $D_1,\dots,D_{10}$ are defined inductively by
$$D_j = (D \setminus (D_1 \cup \dots \cup D_{j-1})) \cap \{ (x,y) : f_j(x,y) > 0 \}, \qquad \text{ for } j=1,\dots,9,$$
and $D_{10} = D \setminus (D_1 \cup \dots \cup D_{9})$.
Due to Fact 1, it is enough to show that
$\mathcal{L}(D_1;5n)=\emptyset$  and $\mathcal{L}(D_k;4n)=\emptyset$ for $k=2,\dots,10$.
For each $D_k$ we will proceed using Fact 3, and this is more or less
a straightforward computation. We present it for $D_1$ (for $D_2,\dots,D_5$ it is very similar)
and $D_6$ (for $D_7,\dots,D_{10}$ it is also very similar).

The set $D_1$ is given by equations $-x-y+5n-1+\varepsilon > 0$, $x \geq 0$, $y \geq 0$ and
it is in fact a simplex (triangle) with $5n$ lattice points along the bottom border line, so
the assumptions of Fact 3 are satisfied.
\begin{figure}[ht!]
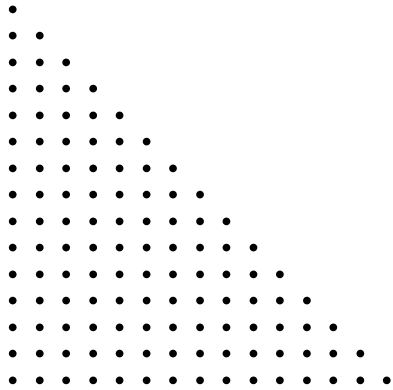

$$
\centertexdraw{
\drawdim pt
\move(0 0)
\fcir f:0 r:1.5
\move(10 0)
\fcir f:0 r:1.5
\move(20 0)
\fcir f:0 r:1.5
\move(30 0)
\fcir f:0 r:1.5
\move(40 0)
\fcir f:0 r:1.5
\move(50 0)
\fcir f:0 r:1.5
\move(60 0)
\fcir f:0 r:1.5
\move(70 0)
\fcir f:0 r:1.5
\move(80 0)
\fcir f:0 r:1.5
\move(90 0)
\fcir f:0 r:1.5
\move(100 0)
\fcir f:0 r:1.5
\move(110 0)
\fcir f:0 r:1.5
\move(120 0)
\fcir f:0 r:1.5
\move(130 0)
\fcir f:0 r:1.5
\move(140 0)
\fcir f:0 r:1.5
\move(0 10)
\fcir f:0 r:1.5
\move(10 10)
\fcir f:0 r:1.5
\move(20 10)
\fcir f:0 r:1.5
\move(30 10)
\fcir f:0 r:1.5
\move(40 10)
\fcir f:0 r:1.5
\move(50 10)
\fcir f:0 r:1.5
\move(60 10)
\fcir f:0 r:1.5
\move(70 10)
\fcir f:0 r:1.5
\move(80 10)
\fcir f:0 r:1.5
\move(90 10)
\fcir f:0 r:1.5
\move(100 10)
\fcir f:0 r:1.5
\move(110 10)
\fcir f:0 r:1.5
\move(120 10)
\fcir f:0 r:1.5
\move(130 10)
\fcir f:0 r:1.5
\move(0 20)
\fcir f:0 r:1.5
\move(10 20)
\fcir f:0 r:1.5
\move(20 20)
\fcir f:0 r:1.5
\move(30 20)
\fcir f:0 r:1.5
\move(40 20)
\fcir f:0 r:1.5
\move(50 20)
\fcir f:0 r:1.5
\move(60 20)
\fcir f:0 r:1.5
\move(70 20)
\fcir f:0 r:1.5
\move(80 20)
\fcir f:0 r:1.5
\move(90 20)
\fcir f:0 r:1.5
\move(100 20)
\fcir f:0 r:1.5
\move(110 20)
\fcir f:0 r:1.5
\move(120 20)
\fcir f:0 r:1.5
\move(0 30)
\fcir f:0 r:1.5
\move(10 30)
\fcir f:0 r:1.5
\move(20 30)
\fcir f:0 r:1.5
\move(30 30)
\fcir f:0 r:1.5
\move(40 30)
\fcir f:0 r:1.5
\move(50 30)
\fcir f:0 r:1.5
\move(60 30)
\fcir f:0 r:1.5
\move(70 30)
\fcir f:0 r:1.5
\move(80 30)
\fcir f:0 r:1.5
\move(90 30)
\fcir f:0 r:1.5
\move(100 30)
\fcir f:0 r:1.5
\move(110 30)
\fcir f:0 r:1.5
\move(0 40)
\fcir f:0 r:1.5
\move(10 40)
\fcir f:0 r:1.5
\move(20 40)
\fcir f:0 r:1.5
\move(30 40)
\fcir f:0 r:1.5
\move(40 40)
\fcir f:0 r:1.5
\move(50 40)
\fcir f:0 r:1.5
\move(60 40)
\fcir f:0 r:1.5
\move(70 40)
\fcir f:0 r:1.5
\move(80 40)
\fcir f:0 r:1.5
\move(90 40)
\fcir f:0 r:1.5
\move(100 40)
\fcir f:0 r:1.5
\move(0 50)
\fcir f:0 r:1.5
\move(10 50)
\fcir f:0 r:1.5
\move(20 50)
\fcir f:0 r:1.5
\move(30 50)
\fcir f:0 r:1.5
\move(40 50)
\fcir f:0 r:1.5
\move(50 50)
\fcir f:0 r:1.5
\move(60 50)
\fcir f:0 r:1.5
\move(70 50)
\fcir f:0 r:1.5
\move(80 50)
\fcir f:0 r:1.5
\move(90 50)
\fcir f:0 r:1.5
\move(0 60)
\fcir f:0 r:1.5
\move(10 60)
\fcir f:0 r:1.5
\move(20 60)
\fcir f:0 r:1.5
\move(30 60)
\fcir f:0 r:1.5
\move(40 60)
\fcir f:0 r:1.5
\move(50 60)
\fcir f:0 r:1.5
\move(60 60)
\fcir f:0 r:1.5
\move(70 60)
\fcir f:0 r:1.5
\move(80 60)
\fcir f:0 r:1.5
\move(0 70)
\fcir f:0 r:1.5
\move(10 70)
\fcir f:0 r:1.5
\move(20 70)
\fcir f:0 r:1.5
\move(30 70)
\fcir f:0 r:1.5
\move(40 70)
\fcir f:0 r:1.5
\move(50 70)
\fcir f:0 r:1.5
\move(60 70)
\fcir f:0 r:1.5
\move(70 70)
\fcir f:0 r:1.5
\move(0 80)
\fcir f:0 r:1.5
\move(10 80)
\fcir f:0 r:1.5
\move(20 80)
\fcir f:0 r:1.5
\move(30 80)
\fcir f:0 r:1.5
\move(40 80)
\fcir f:0 r:1.5
\move(50 80)
\fcir f:0 r:1.5
\move(60 80)
\fcir f:0 r:1.5
\move(0 90)
\fcir f:0 r:1.5
\move(10 90)
\fcir f:0 r:1.5
\move(20 90)
\fcir f:0 r:1.5
\move(30 90)
\fcir f:0 r:1.5
\move(40 90)
\fcir f:0 r:1.5
\move(50 90)
\fcir f:0 r:1.5
\move(0 100)
\fcir f:0 r:1.5
\move(10 100)
\fcir f:0 r:1.5
\move(20 100)
\fcir f:0 r:1.5
\move(30 100)
\fcir f:0 r:1.5
\move(40 100)
\fcir f:0 r:1.5
\move(0 110)
\fcir f:0 r:1.5
\move(10 110)
\fcir f:0 r:1.5
\move(20 110)
\fcir f:0 r:1.5
\move(30 110)
\fcir f:0 r:1.5
\move(0 120)
\fcir f:0 r:1.5
\move(10 120)
\fcir f:0 r:1.5
\move(20 120)
\fcir f:0 r:1.5
\move(0 130)
\fcir f:0 r:1.5
\move(10 130)
\fcir f:0 r:1.5
\move(0 140)
\fcir f:0 r:1.5
}
$$
\caption{The set $D_{1}$ for $n=3$}
\end{figure}

The set $D_6$ is given by
$3x-y-15n-1 + \varepsilon > 0$, $x+y-13n \leq 0$, $x-y-5n-1+\varepsilon < 0$.
\begin{figure}[ht!]
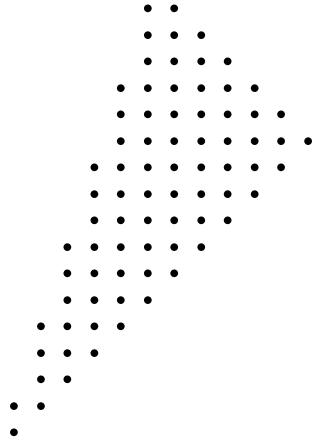

$$
\centertexdraw{
\drawdim pt
\move(160 10)
\fcir f:0 r:1.5
\move(160 20)
\fcir f:0 r:1.5
\move(170 20)
\fcir f:0 r:1.5
\move(170 30)
\fcir f:0 r:1.5
\move(180 30)
\fcir f:0 r:1.5
\move(170 40)
\fcir f:0 r:1.5
\move(180 40)
\fcir f:0 r:1.5
\move(190 40)
\fcir f:0 r:1.5
\move(170 50)
\fcir f:0 r:1.5
\move(180 50)
\fcir f:0 r:1.5
\move(190 50)
\fcir f:0 r:1.5
\move(200 50)
\fcir f:0 r:1.5
\move(180 60)
\fcir f:0 r:1.5
\move(190 60)
\fcir f:0 r:1.5
\move(200 60)
\fcir f:0 r:1.5
\move(210 60)
\fcir f:0 r:1.5
\move(180 70)
\fcir f:0 r:1.5
\move(190 70)
\fcir f:0 r:1.5
\move(200 70)
\fcir f:0 r:1.5
\move(210 70)
\fcir f:0 r:1.5
\move(220 70)
\fcir f:0 r:1.5
\move(180 80)
\fcir f:0 r:1.5
\move(190 80)
\fcir f:0 r:1.5
\move(200 80)
\fcir f:0 r:1.5
\move(210 80)
\fcir f:0 r:1.5
\move(220 80)
\fcir f:0 r:1.5
\move(230 80)
\fcir f:0 r:1.5
\move(190 90)
\fcir f:0 r:1.5
\move(200 90)
\fcir f:0 r:1.5
\move(210 90)
\fcir f:0 r:1.5
\move(220 90)
\fcir f:0 r:1.5
\move(230 90)
\fcir f:0 r:1.5
\move(240 90)
\fcir f:0 r:1.5
\move(190 100)
\fcir f:0 r:1.5
\move(200 100)
\fcir f:0 r:1.5
\move(210 100)
\fcir f:0 r:1.5
\move(220 100)
\fcir f:0 r:1.5
\move(230 100)
\fcir f:0 r:1.5
\move(240 100)
\fcir f:0 r:1.5
\move(250 100)
\fcir f:0 r:1.5
\move(190 110)
\fcir f:0 r:1.5
\move(200 110)
\fcir f:0 r:1.5
\move(210 110)
\fcir f:0 r:1.5
\move(220 110)
\fcir f:0 r:1.5
\move(230 110)
\fcir f:0 r:1.5
\move(240 110)
\fcir f:0 r:1.5
\move(250 110)
\fcir f:0 r:1.5
\move(260 110)
\fcir f:0 r:1.5
\move(200 120)
\fcir f:0 r:1.5
\move(210 120)
\fcir f:0 r:1.5
\move(220 120)
\fcir f:0 r:1.5
\move(230 120)
\fcir f:0 r:1.5
\move(240 120)
\fcir f:0 r:1.5
\move(250 120)
\fcir f:0 r:1.5
\move(260 120)
\fcir f:0 r:1.5
\move(270 120)
\fcir f:0 r:1.5
\move(200 130)
\fcir f:0 r:1.5
\move(210 130)
\fcir f:0 r:1.5
\move(220 130)
\fcir f:0 r:1.5
\move(230 130)
\fcir f:0 r:1.5
\move(240 130)
\fcir f:0 r:1.5
\move(250 130)
\fcir f:0 r:1.5
\move(260 130)
\fcir f:0 r:1.5
\move(200 140)
\fcir f:0 r:1.5
\move(210 140)
\fcir f:0 r:1.5
\move(220 140)
\fcir f:0 r:1.5
\move(230 140)
\fcir f:0 r:1.5
\move(240 140)
\fcir f:0 r:1.5
\move(250 140)
\fcir f:0 r:1.5
\move(210 150)
\fcir f:0 r:1.5
\move(220 150)
\fcir f:0 r:1.5
\move(230 150)
\fcir f:0 r:1.5
\move(240 150)
\fcir f:0 r:1.5
\move(210 160)
\fcir f:0 r:1.5
\move(220 160)
\fcir f:0 r:1.5
\move(230 160)
\fcir f:0 r:1.5
\move(210 170)
\fcir f:0 r:1.5
\move(220 170)
\fcir f:0 r:1.5
}
$$
\caption{The set $D_{6}$ for $n=3$}
\end{figure}

Let $(x,y) \in D_6$. Observe that for $x > 9n$ we would have $y < 4n$, so $-y>-4n$ and
$x-y-5n-1>9n-4n-5n-1>-1$, hence $x-y-5n-1 \geq 0$, a contradiction.
For $x\leq 5n$ we would have $y \leq 3x-15n-1 \leq -1$, a contradiction. Hence
$x\in[5n+1,9n]$, so $D_6$ lies on at most $4n$ vertical lines.
From the the first and the last defining inequality we easily obtain that for
$(x,y) \in D_6$, $x \in [5n+1,7n]$ we must have $y \in [x-5n,3x-15n-1]$.
But for $x=5n+1$ we have exactly two lattice points in the interval
$[x-5n,3x-15n-1]=[1,2]$. If $x$ increases by one then
$\#[x-5n,3x-15n-1]$ inceases by two, so we have at most
$2,4,\dots,4n-2,4n$ points on $2n$ vertical lines $\ell_2 = \{x=5n+1\}, \ell_4 = \{x=5n+2\},\dots, \ell_{4n} = \{x=7n\}$.
Similarly we show that on lines $\ell_{4n-1} = \{x=7n+1\}, \ell_{4n-3} = \{x=7n+2\},\dots, \ell_1 = \{x=9n\}$ we have at most
$4n-1,4n-3,\dots,1$ points.

Facts 1 and 2 can be understood from the point of view of toric degenerations,
and as a result we can get a more geometric proof of non-semi-effectivity of
$13L-5E_1-4E_2- \dots -4E_{10}$. The translation of Fact 1 into a toric statement
is based on the construction of a projective toric variety
from any polytope $D$ (defined as an intersection of half-spaces)
in a space $M$ as follows.

Let $Q$ be a face of $D$.
Define a cone $C_Q$ in the dual space $N$ by
$$
C_Q = \left\{v \in N \;|\; \langle v,p-q \rangle \geq 0 \mbox{ for all } p \in D \mbox{ and } q \in Q \right\}
$$
As $Q$ varies over all of the faces of $D$,
one gets a fan of cones $F_D$,
which defines a toric variety $X_D$
(see \cite{Fulton93}, Section 1.5).
Furthermore, this toric variety comes equipped with an ample line bundle $L_D$
(see \cite{Fulton93}, Section 3.4, page 72).

For example,
if the polytope is the interval $[0,d]$ in one-dimensional space,
then the variety is the projective line,
and the line bundle has degree $d$.
Similarly, if the polytope is the usual triangle
with vertices $(0,0)$, $(0,d)$, $(d,0)$ in the plane,
then the variety is the projective plane and the line bundle has degree $d$.

The construction works up to a point with unbounded polytopes, too.
The one that is useful for us is to take a polytope $D$
and cross it with the positive reals.
This gives a new polytope, $D'$,
and the associated toric variety is $X_D \times \mathbb{A}^1$,
the product of the original toric variety for $D$ with the affine line.
The line bundle is just the pullback of the line bundle on $X_D$.

A more interesting example of an unbounded polytope
can be used to see a degeneration.
Take the set of three line segments:
$(0,0)$ to $(1,-1)$;
$(1,-1)$ to $(2,-1)$; and
$(2,-1)$ to $(3,0)$.
Take as the (unbounded) polytope all points $(x,y)$
which lie on or above these line segments:
$$
D = \left\{(x,y) \;|\; 0\leq x\leq 3;
\begin{array}{lcl}
\mbox{ if } 0 \leq x \leq 1 & \mbox{ then } & y \geq -x;\\
\mbox{ if } 1 \leq x \leq 2 & \mbox{ then } & y \geq -1;\\
\mbox{ if } 2 \leq x \leq 3 & \mbox{ then } & y \geq x-3
\end{array}
\right\}
$$
Each of the four vertices of $D$
gives four maximal cones of the fan in the dual space,
and these are bounded by the rays
$(-1,0)$ and $(-1,1)$;
$(-1,1)$ and $(0,1)$;
$(0,1)$ and $(1,1)$;
$(1,1)$ and $(1,0)$.
The associated toric variety is the blow-up of $\mathbb{P}^1 \times \mathbb{A}^1$
at two points of the central fiber,
giving a degeneration of $\mathbb{P}^1$ to a chain of three $\mathbb{P}^1$'s.
As for the line bundle, it restricts to degree $3$ on the general fiber,
and degree one on each component of the special fiber.
This then realizes the degeneration of the twisted cubic curve
to a chain of three lines.

What we see in this example is the $3$-Veronese of $\mathbb{P}^1$,
defined by the polytope which is just the interval $[0,3]$,
and a `subdivision' of the interval
suitably used to create the degeneration.
In order to properly construct the degeneration,
we had to create
the `lifting' of the subdivision to one dimension higher.

In general, one may in fact use a similar construction.
One needs a similar lifting of the subdivision
of the `bottom' faces of an unbounded polytope.
The key ingredient
is to have that the faces of the unbounded polytope
lie exactly above the sub-polytopes.
This is the idea of a \emph{regular} subdivision.

Suppose $D$ is a polytope, and $S=\bigcup D_i$ is a subdivision of $D$.
Suppose that we have a real-valued continuous function $F$ on $D$,
which is linear on each $D_i$.
One says that the function $F$ is \emph{strictly $S$-convex}
if it is a convex function,
with the additional property that for any subpolytope $D_i$ of $S$,
if $L_i$ is the linear function for the subpolytope,
extended to the entire space,
then $F(p) > L_i(p)$ for all $p$ in $D\setminus D_i$.

A subdivision $S$ of $D$ is called \emph{regular}
if $D$ has a strictly $S$-convex function.
One may consult \cite{Ziegler95}, Chapter 5,
for additional detail on these ideas.

Now suppose we have a polytope $D$ in $\mathbb{R}^n$,
defining a toric variety $X_D$, as above.
Suppose we have a subdivision $S$ of $D$,
which is regular;
let $F$ be a strictly $S$-convex function on $D$.
Define the unbounded polytope $P(F)$ in $\mathbb{R}^{n+1}$ by:
$$
P(F) = \left\{(x,y) \in P \times \mathbb{R}^1 \;|\; y \geq F(x) \right\}
$$

This polytope $P(F)$ defines a toric variety $Y$,
which is a suitable blow-up of $X_D \times \mathbb{A}^1$ in the central fiber,
and exhibits a degeneration of $X_D$
to a union of toric varieties defined by the subpolytopes $D_i$.

There is a more elementary way of seeing the degeneration,
embedded in projective space.
Suppose that $D$ is a polygon
(with lattice points for vertices of course)
whose set of lattice points is $\left\{m_i\right\}$.
These lattice points correspond to monomials in the variables in the usual way:
the lattice coordinates are the exponents of the monomials.
If there are $k+1$ monomials,
this defines a mapping $g: \mathbb{P}^2 \to \mathbb{P}^k$
by sending a point $[x]$ to the point $[x^{m_i}]$,
and $X_D$ is the image of this mapping.

Now suppose we have this subdivision $S$,
with a strictly $S$-convex function $F$ as above.
If $y$ is a coordinate on $\mathbb{A}^1$,
this data defines a new mapping
$$G: \mathbb{P}^2 \times (\mathbb{A}^1-{0}) \to \mathbb{P}^k$$
by sending a pair $([x],y)$ to the point $[y^{F(m_i)}x^{m_i}]$.

For a fixed $y$ not equal to zero,
this is simply a scaling of the original map $g$,
and so the $\mathbb{P}^2$ fibers are all mapped to $X_D$.
We want to know what happens in the limit as $y$ approaches zero.
Of course, the values of $F$ may well be either negative, or all positive,
and so we cannot take the limit with the above formula.

To see what happens, consider one of the subpolytopes $D_i$,
and the corresponding linear function $L=L_i$.
Vary the above formula by multiplying through by $y^{-L}$.
This will have the effect that as $y$ approaches zero,
a limiting value will be available in the projective space.

The simplest example of this construction is to take a polytope $D$
and divide it into two pieces, by a hyperplane.
Suppose that this hyperplane does not pass through any of the lattice points.
Then one can start with a convex function $H$ that is linear on the two sides
of the hyperplane.
Form the construction above,
and then take the convex hull of the graph of $H$,
using the lattice points.
This will produce two `primary' subpolytopes as faces of this convex hull,
containing the two subsets of the lattice points,
on the two sides of the hyperplane.
It will also contain a set of `secondary' subpolytopes
that cross the boundary of the hyperplane.

The construction then produces a toric degeneration,
into two `primary' toric varieties, and a set of `secondary' ones.
If we want to use this construction to understand an interpolation
problem \`a la Dumnicki,
we put the limits of the fat points on the two `primary' varieties,
and none on the secondary ones.
If the interpolation problems on the two `primary' varieties
give empty linear systems, then the linear system will be empty
in the limit, and therefore empty on the general fiber
(by semicontinuity).

In Dumnicki's example, with the triangle subdivided into ten subpolytopes,
one easily sees that this subdivision can be achieved iteratively,
by making a single-hyperplane subdivision, nine times.
This should give the result,
and provide a `toric' interpretation for the example.

The proof of Fact 2 mainly involves an interpolation matrix $M(D)$.
Let $(X_D, L_D)$ be the polarized toric surface defined by  the polygon $D.$ Recall that $X_D=\cup _{v\in\text{Vert}(D)}U_v$ where $U_v$ denotes an affine neighborhood around the fixed point on $X_D$ corresponding to $v.$  By placing one vertex of the polytope at $0$ one chooses an affine patch around the corresponding fixed point  so that:
$$H^0(U_v,L_D|_{U_v})=\bigoplus_{\alpha=(a,b)\in D}\A\langle s_\alpha \rangle \text{, where }s_\alpha=x^ay^b.$$

Let $J_m(D)=J_m(L_D)$ be the $m^{th}$ jet sheaf, i.e., the coherent sheaf defined locally at smooth points $p,$ as
$J_m(L_D)_p=H^0(L_D\otimes {\cal O}_{X_D}/{\mathfrak m}_p^{m+1}),$ where $\mathfrak{m}_p$ denotes the maximal ideal at $p.$
The $m^{th}$ evaluation map at a smooth point $p$ is defined as the map:
$$\nu_{k,p}:H^0(X_D,L_D)\to J_m(L_D)_p$$
that assigns  $\nu_{k,p}(s_i)=(s_i(p),\ldots, \frac{\partial^{i+j} s_i}{\partial x^j\partial x^i}(p),\ldots)_{1\leq i+ j\leq m}$ to a basic section $s_i,$
for a choice of local coordinates $(x,y)$ in a  neighborhood $U_v$ around $p.$ The map is in fact independent of the choice of $U_v$ and $Im(\nu_{k,p})=\Osc^m_p$ is referred to as the $m^{th}$ osculating space at $p.$

Recall that the dimension of $\Osc^m_p$ at a generic smooth point, e.g., at the generic point in the torus $T\cong (\C^*)^2\subset X_D,$ is constant and it is called the generic osculation dimension.

For a smooth point $p$ the dimension is:
$$\dim(\shl(D,m)_{p})=\dim(X_D,L_D)-\dim(\Osc^{m-1}_p).$$
We denote by $1=(1,1)$ the generic point of the torus. Because $\dim(\Osc^m_1)\geq\dim( \Osc^m_p)$ for all other smooth points $p,$
$$\dim(\shl(D,m))=\dim(\shl(D,m)_{1})\text{ and }$$
$$\shl(D,m)=\emptyset \text{ if and only if }\dim\Osc_1^{m-1}=\dim(X_D,L_D).$$
Let $\dim(H^0(X_D,L_D))=|D|=N.$
The map $\nu_{m-1,1}:H^0(X_D,L_D)\to J_{m-1}(L_D)_1$ is represented by the $N\times {m+1\choose 2}$ matrix
 $M(D),$  whose rows are indexed by the derivatives, until order $m-1$, with respect to $x$ and $y$, of the $s_i$ evaluated at $1$ while the columns  are indexed by monomials in $D$.  Thus, for example, the $(i,j)$ entry in $M(D)$ is the derivative, in row $i$, of the monomial in column $j$.
 \[ M(D)=\left(\begin{array}{cccc}
 1&1&\ldots&1\\
a_1&a_2&\ldots& a_N\\
b_1&b_2&\ldots& b_N \\
 a_1b_1&a_2b_2&\ldots& a_Nb_N\\
 \ldots&\ldots&\ldots&\ldots\end{array}\right)\]
 where $D=\{(a_1,b_1),\ldots,(a_N,b_N)\}.$

 In the case when $|D|= \binom{m+1}{2}$ the matrix is a square matrix and thus one has ${\mathcal{L}}(D,m)=\emptyset$  if and only if $\det(M(D))\not=0$, from which the existence of a polynomial $F \in   {\mathbb{Q}}[x,y]$ with $\deg(F)=m-1$, such that $F(a,b)=0$ for all $(a,b) \in D$ is derived. It is important to remark here that the proof gives an interpretation of the lattice points given by $D$ as points in $\A^2$.

We observed that the set of derivatives used can be substituted by
any set $E=\{\partial^{\alpha_1}_x\partial^{\beta_1}_y, \dots,
\partial^{\alpha_N}_x\partial^{\beta_N}_y\}$ which
(interpreting $(\alpha_i,\beta_i)$ as points in $\mathbb{N}^2$) is closed
under downward and leftward moves; in the literature such sets
$E$ are known as  \emph{staircases} (see \cite{Eva07}) and they are used
to define monomial ideals $I(E)=(x^\alpha y^\beta)_{(\alpha,\beta)\not\in E}$. In the method
above one then has to use
a modified jet sheaf where ${\mathfrak m}_p^{m+1}$ is replaced by
the translation $I(E)_p$ of $I(E)$ to $p$ by the action of the torus.
Denoting by $\mathcal L (D,E)$ the subsystem of $\mathcal L(D)$
consisting of sections which belong to the monomial ideal $\mathcal I(E)_1$
supported at the generic point, we obtain a
\vskip0.2cm
{\bf Generalized Fact 2:}
Given $D \subset {\mathbb{N}}^2$  such that $|D|=|E|$, then  ${\mathcal{L}}(D,E)\ne\emptyset$ if and only if there exists a nonzero
 polynomial $F \in   {\mathbb{Q}}[x,y]$ containing only polynomials of $E$, such that $F(a,b)=0$ for all $(a,b) \in D$.
\vskip0.2cm
Such a generalization allows to deal with the \emph{Hermite interpolation
schemes of tree type} considered by Lorentz, (see \cite{Lor00}, Section 3), which
is actually a specialized case of the general monomial interpolation problem
suggested as Problem \ref{pro:mon}. Note that in the original problem the coordinates
in which the scheme is monomial can be ``deformed'' by the immersion of the
scheme in the surface, whereas here they are torically fixed. Thus
when ${\mathcal{L}}(D,E)=\emptyset$ it follows that
${\mathcal{L}}(D,\mathcal I_Y)=\emptyset$ for general schemes $Y$
isomorphic to $\operatorname{Spec}k[x,y]/I_E$, but not conversely.

\subsection{Linear systems connected to hyperplane arrangements}
We keep the notation introduced in Section 2.7.
Results of Falk-Yuzvinsky \cite{FY} show that certain combinatorial objects
known as weak $(k,m)$--multinets give rise to components of $R^1(\mathcal A)$,
and \cite{Sch} shows that a weak $(k,m)$--multinet corresponds to a divisor $A$ on
$X$ with $h^0(A)=2$. If the weak multinet is actually a net,
then $D_{\mathcal A} = A+B$, with $h^0(B) \ge km-{m+1 \choose 2}.$
This decomposition then gives rise to determinantal equations (and syzygies)
in the Orlik-Terao ideal. In particular, since $h^0(A)=2$, if $h^0(B)=m$,
then $I$ contains the two by two minors of a matrix of linear forms, which
has an Eagon-Northcott resolution, which has a very explicit description.
The next example shows that not all linear first syzygies arise from
components of $R^1(\mathcal A)$.
\begin{example}\rm
The arrangement below is obtained by deleting a line
from the Maclane arrangement (Example 10.7 of \cite{SU}).
\newline
\begin{figure}[ht]
\begin{center}
\label{fig:dx2}%
\begin{minipage}[t]{0.3\textwidth}
\setlength{\unitlength}{8.5pt}
\begin{picture}(4.5,9.2)(-2,-3)
\multiput(-2,1)(0,2){2}{\line(1,0){8}}
\multiput(1,-2)(2,0){2}{\line(0,1){8}}
\put(-2,4){\line(1,-1){6}}
\put(-1.5,7.5){\line(1,-1){7.5}}
\put(-1.5,8.0){\line(1,-2){5.5}}%
\put(4.7,-4.0){\makebox(0,0){$0$}}
\put(4.7,-2.3){\makebox(0,0){$1$}}
\put(6.7,-0.8){\makebox(0,0){$2$}}
\put(6.7,1){\makebox(0,0){$3$}}
\put(6.7,3){\makebox(0,0){$4$}}
\put(3,6.9){\makebox(0,0){$5$}}
\put(1,6.9){\makebox(0,0){$6$}}
\end{picture}
\end{minipage}
\end{center}
\caption{\textsf{The $M_8^-$ arrangement}}
\label{fig:dMac}
\end{figure}

The graded betti numbers for the Orlik-Terao algebra are
$$
\vbox{\offinterlineskip 
\halign{\strut\hfil# \ \vrule\quad&# \ &# \ &# \ &# \ &# \ &# \ &# \ &# \ &# \ &# \ &# \ &# \ &# \
 \cr
total&1&8&36&56&35&8 \cr \noalign {\hrule} 0&1 &--&--& --&--  \cr 1&--&7 &1 &--&--&--   \cr 2&--&1
 &35 &56 &35&8  \cr \noalign{\bigskip} \noalign{\smallskip} }}
$$
The entry in position $(i,j)$ is $\mbox{dim}_{\mathbb{C}}Tor_i^R(C(\mathcal{A}),\mathbb{C})_{i+j}$,
so there are seven quadratic generators for $I$, and a single linear first
syzygy. On the other hand, for this arrangement, $R^1(\mathcal A)$ consists only of
local components.

An analysis shows that we may choose a basis for the quadratic component
$I_2$ so that the linear first syzygy involves only five of the seven
minimal quadratic generators. Letting $J$ denote the ideal
generated by these five elements, we compute that $J$ has betti diagram
$$
\vbox{\offinterlineskip 
\halign{\strut\hfil# \ \vrule\quad&# \ &# \ &# \ &# \ &# \ &# \
&# \ &# \ &# \ &# \ &# \ &# \ &# \
\cr
total&1&5&12&10&2\cr
\noalign {\hrule}
0&1 &--&--& --&--&    \cr
1&--&5 &1 & --&--&        \cr
2&--&-- &11 & 10&1&        \cr
3&--&-- &-- &--&1&         \cr
\noalign{\bigskip}
\noalign{\smallskip}
}}
$$
The corresponding variety is a Cohen-Macaulay surface; intersecting with a generic $\mathbb{P}^5$
yields a smooth, projectively normal curve of genus seven and degree eleven.
This curve appears in \cite{ST}
as a counterexample to several conjectures in algebraic geometry. A sole
linear first syzygy cannot arise from a decomposition of $D_{\mathcal A}$:
such a decomposition would yield at least one additional linear first syzygy.
\end{example}

\subsection{Limitations of multiplier ideal approach to bounds for symbolic powers}

It is natural to ask whether the strategy of \cite{ELS}, using asymptotic multiplier ideals,
can be used to prove Conjecture~\ref{bound-symbolic-powers-1} or any of the improved versions
in Problem~\ref{bound-symbolic-powers-2}, or whether another strategy is needed.
A weaker version of \ref{bound-symbolic-powers-1} was shown by Takagi--Yoshida
using tight closure methods \cite{takagi-yoshida},
but (as observed in \cite{MR2591736}) the same result follows from the asymptotic multiplier ideal
approach of \cite{ELS}.
On the other hand it seems doubtful that the same approach can prove Conjecture~\ref{bound-symbolic-powers-1} in full strength.
So, can the asymptotic multiplier ideals approach prove any of the statements in Problem~\ref{bound-symbolic-powers-2},
or at least weaker versions?
We will see an example suggesting that the answer to this question is negative.

The strategy of \cite{ELS} is as follows.
Let $I$ be an ideal with $\bight(I) = e$.
Consider the \textit{graded system of ideals} $I^{(\bullet)} = \{I^{(p)}\}_{p \in \N}$, the sequence of symbolic powers of $I$.
For each positive real number $t$ there is an ideal associated to this graded system, called the $t^{th}$ asymptotic multiplier ideal
of $I^{(\bullet)}$, and denoted $\J(t \cdot I^{(\bullet)})$.
These ideals enjoy a number of remarkable properties; see \cite{ELS} (where they were introduced)
or \cite{PAG}.
In particular it is shown in \cite{ELS} that the following containments hold:
\[
  I^{(re)} \subseteq \J(re \cdot I^{(\bullet)}) \subseteq \J(e \cdot I^{(\bullet)})^r \subseteq I^r \, .
\]
Here the first and third inclusions follow more or less directly from the definition of multiplier ideals,
while the second inclusion follows from the subadditivity theorem of Demailly--Ein--Lazarsfeld \cite{MR1786484,PAG}.
The weak version of \ref{bound-symbolic-powers-1} follows from the fact that
if $\J(\ell \cdot I^{(\bullet)}) = (1)$, which holds for sufficiently small $\ell$,
then $I^{(re-\ell)} \subseteq \J(re \cdot I^{(\bullet)})$.
(Conjecture~\ref{bound-symbolic-powers-1} would follow if $\ell=e-1$.)
Note that the second and third containments remain the same; only the first is made tighter.

To give a positive answer to Problem~\ref{bound-symbolic-powers-2} by this approach
would require an improvement at the far right-hand side of the above containments.
An easy example suggests this may not be possible.
Let $R = \C[x,y,z]$ and let $I = (xy,xz,yz)$, the ideal defining the union of the three coordinate axes
(in particular $e = \bight(I) = 2$).
In general multiplier ideals are difficult to compute, and few examples are known;
the situation is worse for asymptotic multiplier ideals, where even fewer examples have been computed.
However, because this ideal $I$ is a monomial ideal, it is possible to compute the asymptotic multiplier ideals
appearing in the above containments.
This is carried out in \cite{MR2591736} (following \cite{MR1936888})
with the result that, for each $r$,
\[
  \J(2 \cdot I^{(\bullet)})^r = I^r \, .
\]
That is, it is not possible to improve the third containment while leaving the first and second the same.

It is still possible that an improvement might be attainable for some ideals, just not this particular one.
However at this point the asymptotic multiplier ideals approach does not seem promising.

\appendix
\section{Logarithmic differentials and the Miyaoka--Yau inequality}\label{appendixA}

\subsection{Basics}
Here we provide a very quick overview of logarithmic differentials, in particular, no proofs are given. For more complete discussions the reader should consult \cite{EV} or \cite{Iitaka}, for example.

Let $X$ be a smooth projective variety of dimension $n$ over the complex numbers. A reduced divisor $D=D_1+\dots+D_r$ is called a \emph{simple normal crossing divisor} if all of its irreducible components are smooth, and every point of the support of $D$ has an open neighborhood $U$ sucht that $D$ restricted to $U$ looks like an intersection of coordinate hyperplanes.  In particular, if $D$ is prime, then it is forced to be smooth.

With $(X,D)$ as above, let $x\in D$, and $D_1,\dots,D_s$ be the irreducible components of $D$ containing the point $x$.  By definition there exists an open neighborhood $x\in U\subseteq X$, and a local coordinate system $x_1,\dots,x_n$ at $x$ on $U$ such that
\[
 D_i\cap U \equ V(x_i)
\]
for all $1\leq i\leq s$. We call such a coordinate system a logarithmic coordinate system at $x$ along $D$.

In what follows, $D$ denotes a simple normal crossing divisor. At first we look at the following heuristic picture.
We will look at  various interesting subbundles inside the tangent bundle $T_X$ of $X$. In order to give some  geometric intuition, we will often switch to the language of differential geometry.

With notation as above,
\[
T_X\otimes\shi_{D/X} \dsubseteq  T_X
\]
is a sub-vector bundle, that can be identified with vector fields on $X$ fixing $D$ pointwise. On the other hand, it
is natural to consider the subbundle corresponding to vector fields that stabilize the divisor $D$. This latter is denoted by
$T_X(-\log D)$, and is again a sub-vector bundle of $T_X$. Summarizing, one has a sequence of inclusions of vector bundles
\[
 T_X\otimes \shi_{D/X} \dsubseteq T_X(-\log D) \dsubseteq T_X\ .
\]
With the help of a little homological algebra one obtains a dual sequence of inclusions
\[
(T_X)^* \dsubseteq (T_X(-\log D))^* \dsubseteq (T_X\otimes\shi_{D/X})^*.
\]
Observe that
\[
 (T_X)^* \simeq \Omega^1_X\ ,\ (T_X\otimes\shi_{D/X})^*\simeq \Omega^1_X\otimes \OO_X(D) .
\]
In fact, we can write
\[
\Omega^1_X \dsubseteq \Omega^1_X(\log D) \dsubseteq \Omega^1_X\otimes \OO_X(D),
\]
where $(T_X(-\log D))^* \simeq \Omega^1_X(\log D)$, and the latter is called the sheaf of logarithmic differentials with poles along $D$. The formal definition is as follows.

\begin{definition}
With notation as above, the \emph{sheaf of logarithmic differentials with poles along $D$}, $\Omega^1_X(\log D)$, is the $\OO_X$-submodule of $\Omega^1_X\otimes_{\OO_X} \OO_X(D)$ determined uniquely by the following properties:
\begin{enumerate}
 \item $\Omega^1_X(\log D)|_{X-D} \simeq \Omega^1_{X-D}$,
\item  Let $U\subseteq X$ be an open subset intersecting $D$.  For an element $f\in (\Omega^1_X\otimes\OO_X(D))(U)$,
\[
 f\in \Omega^1_X(\log D)(U) \text{\ \ if and only if \ \ } f_x\equ \sum_{i=1}^{s}g_i\frac{dx_i}{x_i}+\sum_{i=s+1}^{n}h_i dx_i
\]
for every (closed) point $x\in D\cap U$, and every  logarithmic coordinate system $x_1,\dots,x_n$ at $x$ along $D$, where $g_i,h_i\in \OO_{X,x}$.
\end{enumerate}
Next, we set
\[
\Omega^p_X(\log D) := \wedge^p \Omega^1_X(\log D).
\]
\end{definition}

\begin{remark}\label{rmk:dual}\rm
   For a smooth projective variety $X$ of dimension $n$ and a simple normal
   crossing divisor $D$ on $X$ we have the duality
   $$(\Omega^j(\log D))^* = \Omega^{n-j}(\log D)\otimes \OO_X(-K_X-D).$$
   In particular, for a surface $X$,
   $$(\olc)^* \simeq \olc\otimes\OO_X(-K_X-C),$$
   hence also
   $$(\Sym^m\olc)^*\simeq \Sym^m\olc\otimes\OO_X(-m(K_X+C)).$$
   Here $C\subseteq X$ denotes a simple normal crossing curve on the surface $X$.
\end{remark}

One of the fundamental tools for computing with logarithmic differentials is the collection of short exact sequences
\[
\ses{\Omega_X^p}{\Omega_X^p(\log D)}{\Omega_D^{p-1}}
\]
for all $p\geq 1$, in particular,
\begin{equation}\label{eqn:first}
\ses{\Omega_X^1}{\Omega_X^1(\log D)}{\OO_D}.
\end{equation}

   From now on we assume that  $X$ is a smooth projective surface and $C\subseteq X$ is an irreducible curve.
   We will compute the Chern classes of the bundle $\shf:= \olc$. Since $\rk \shf =2$, we only care about $c_1$ and $c_2$. From the Whitney sum formula \cite[Theorem 3.2 (e)]{Ful84} and the exact sequence \eqnref{eqn:first} we obtain
\[
 c_t(\olc) \equ c_t(\Omega_X^1)\cdot c_t(\OO_C),
\]
and from the short exact sequence
\[
 \ses{\OO_X(-C)}{\OO_X}{\OO_C}
\]
it follows that
\[
 c_t(\OO_C) \equ \frac{1}{c_t(\OO_X(-C))} \equ c_t(\OO_X(C)).
\]
 Hence
\begin{eqnarray*}
 && 1+c_1(\olc)t+c_2(\olc)t^2 \\
 && \hspace*{2cm} \equ  (1+c_1(\Omega_X^1)t+c_2(\Omega_X^1)t^2) \cdot  (1+c_1(\OO_X(C))t),
\end{eqnarray*}
which gives
\begin{eqnarray*}
 c_1(\olc) & = & c_1(\OO_X(C))+c_1(\Omega_X^1)=K_X+C \\
c_2(\olc) & = & c_2(\Omega_X^1)+ c_1(\Omega_X^1)\cdot c_1(\OO_X(C))=c_2(X)+(K_X\cdot C).
\end{eqnarray*}
When it comes to computing various expressions in terms of Chern classes, we can identify $c_1(\OO_X(C))$ with $C$ and $c_1(\Omega_X^1)$ with $K_X$.

\subsection{The Miyaoka--Yau inequality for $\olc$}

The purpose of this section is to establish the logarithmic version of the Miyaoka--Yau inequality originally proved by Sakai \cite[Theorem 7.6]{Sakai}.  Sakai's argument, as will our exposition,  follows closely that of Miyaoka in \cite{Mi}.  Since it is our goal to provide  a clear picture, we will present a reasonably self-contained proof.  The main result is presented in Theorem~\ref{thm:logMY}.

In the course of this section $X$ denotes an arbitrary  smooth complex projective surface unless otherwise stated.
We always work over the complex numbers. Our starting point is  the following fundamental result of Bogomolov (see  \cite[Proposition 2.2]{VdV} for a nice and detailed proof).

\begin{proposition}
Let $\shl\subseteq\Omega_X^1$ be a line bundle, then $\shl$ cannot be big.
\end{proposition}

This result has been generalized to log differential forms (see \cite[Corollary 6.9]{EV}).

\begin{theorem}[Bogomolov-Sommese Vanishing]\label{thm:BS}
Let $X$ be a smooth complex projective variety, $\shl$ a line bundle, $C$ a normal crossing divisor on $X$. Then
\[
 \HH{0}{X}{\Omega^{a}_X(\log C)\otimes \shl^{-1}} \equ 0
\]
 for all $a<\kappa(X,\shl)$.
\end{theorem}

A fundamental geometric consequence of Bogomolov--Sommese vanishing is that sub-line bundles of the sheaf of log differentials
cannot have many sections. More precisely, we have the following statement, which has already made an appearance as Corollary~\ref{cor:notbig}.
We repeat the proof for the reader's convenience.

\begin{corollary}\label{cor:notbig2}
Let $X$ be a smooth projective surface, $C$ a normal crossing divisor on $X$. Then $\olc$ contains no big line bundles.
\end{corollary}
\begin{proof}
Note that $\shl$ is big if and only if $\kappa(X,\shl)=2$. An inclusion $\shl\hookrightarrow\olc$ gives rise to a non-trivial section of $\HH{0}{X}{\olc\otimes \shl^{-1}}$, which vanishes according to Theorem~\ref{thm:BS}.
\end{proof}

As a consequence we obtain numerical criteria for line bundles contained in $\olc$, with $C$ a normal crossing divisor on $X$.

\begin{corollary}\label{cor:num}
Let $\OO_X(D)\subseteq\olc$ be a line bundle, $\OO_X(P)$ a nef line bundle on $X$. Then  either $(P\cdot D)\leq 0$  or  $(D^2)\leq 0$. In particular, if $D$ is effective, then $(D^2)\leq 0$.
\end{corollary}
\begin{proof}
   Assuming  $(P\cdot D)>0$ we will show $(D^2)\leq 0$.
   The linear system $|K_X-mD|= \emptyset$ for large $m\geq 0$ because its intersection
   with the nef divisor $P$ becomes negative for large $m$:
$$
   (K_X-mD)\cdot P=K_X\cdot P-m P\cdot D<0 \;\;\mbox{ for }\; m\gg 0.
$$
   By Serre duality,
\[
 \HH{2}{X}{\OO_X(mD)} \equ \HH{0}{X}{\OO_X(K_X-mD)} \equ 0,
\]
which implies via Riemann Roch that for all $m\gg 0$ we have
\begin{eqnarray*}
\gamma\cdot m &  \geq & \hh{0}{X}{\OO_X(mD)} \geq  \hh{0}{X}{\OO_X(mD)} - \hh{1}{X}{\OO_X(mD)} \\
& \equ &  \chi(X,\OO_X(mD)) \equ \frac{(D^2)}{2}m^2+O(m).
\end{eqnarray*}
The left-most inequality comes from the fact that $D$ is not big by Corollary~\ref{cor:notbig2}. It follows that $(D^2)$ must be non-positive.
\end{proof}

Connecting up with the Hodge Index Theorem gives one of the main technical ingredients of the logarithmic Miyaoka--Yau inequality (see also \cite[Proposition 1]{Mi}). Note that the original statement requires $\det\shf$ to be semi-ample but it is enough for the argument to assume $\det\shf$ is nef. This proposition gives a nice criterion for vanishing of the group of global sections for certain vector bundles.

\begin{proposition}\label{prop:main}
Let $\shf\subseteq \olc$ be a locally free sheaf of rank $2$ with nef  determinant, $D$ an arbitrary divisor on $X$. If
\[
\HH{0}{X}{\shf\otimes\OO_X(-D)} \neq 0
\]
then
\[
 (\det\shf\cdot D) \dleq \max\st{0,c_2(\shf)}.
\]
\end{proposition}
\begin{proof}
Consider the projective bundle $\pi:V:= \P(\shf)\to X$ associated to $\shf$. Denoting by $H$ the tautological line bundle on $V$,
\[
\HH{0}{V}{\OO_V(H-\pi^*D)} \simeq \HH{0}{X}{\shf\otimes\OO_X(-D)} \,\neq\, 0
\]
by assumption. Let $W\in |H-\pi^*D|$, and write it as
\[
 W \equ W_0 + \pi^*D'
\]
where $W_0\sim H-\pi^*(D+D')$ is a prime divisor and $D'$ a suitable effective (possibly trivial) divisor on $X$.
The line bundle $\det\shf$ is nef, therefore
\[
 (\det\shf\cdot D') \dgeq 0.
\]
Setting $D'':= D+D'$, we note that
\[
 (\det\shf\cdot D) \dleq (\det\shf\cdot D''),
\]
 hence it suffices to prove the proposition for $D''$ in place of $D$. An application of the Hodge Index Theorem in the guise of Lemma~\ref{lem:fund} gives that
\begin{equation}\label{eq:fund}
 (\det\shf\cdot D'') \dleq c_2(\shf)+(D''^2).
\end{equation}
The divisor $W_0$ gives rise to  a non-trivial section $s\in \HH{0}{X}{\shf\otimes\OO_X(-D'')}$, which in turn embeds $\OO_X(D'')$ into $\shf\subseteq \olc$. Applying Corollary~\ref{cor:num} with  the nef line bundle $\det\shf$ results in
\[
 (\det\shf\cdot D'') \dleq 0 \text{ or } (D''^2) \dleq 0,
\]
which, coupled with the inequality \eqref{eq:fund}, finishes the proof.
\end{proof}

\begin{lemma}\label{lem:fund}
With notation as above, let $W_0\in |H-\pi^*D''|$ be an irreducible divisor on $V$. Then
\[
 (\det\shf\cdot D) \dleq c_2(\shf)+(D''^2)\ .
\]
\end{lemma}
\begin{proof}
This is \cite[Lemma 8]{Mi}.
\end{proof}

   The following theorem, which is \cite[Theorem 3]{Mi}, generalizes Proposition \ref{prop:main}
   to symmetric powers. Its proof is based on the so called `branched covering trick'
   \cite[Theorem I.18.2]{BPV} and Remark \ref{rmk:dual}.

\begin{theorem}\label{thm:3}
Let $\shf\subseteq \olc$ be a locally free sheaf of rank $2$  with nef determinant bundle. If
\[
 \HH{0}{X}{\Sym^m\shf\otimes\OO_X(-D)} \,\neq\, 0
\]
for some positive integer $m$, then
\[
(\det\shf\cdot D) \dleq \max\st{0,mc_2(\shf)}.
\]
\end{theorem}
\begin{proof}
Let $s\in \HH{0}{X}{\Sym^m\shf\otimes\OO_X(-D)}$ be a non-trivial section. By the aforementioned branched covering trick there exists a covering
$f:Y\to X$ such that
\[
 f^*(s) \in \HH{0}{Y}{\Sym^mf^*\shf\otimes\OO_Y(-f^*D)}
\]
 can be written as
\[
 f^*(s) \equ s_1\dots s_m
\]
with $s_i\in \HH{0}{Y}{f^*\shf\otimes\OO_Y(-D_i)}$. Also note that $(\det f^*\shf)^{\otimes m} \simeq f^*(\det\shf)^{\otimes m}$ is  nef, and
that we have canonical injections
\[
f^*\shf \dsubseteq f^*\olc \dsubseteq \Omega_Y^1(\log f^{-1}(C)).
\]
The second containment comes from \cite[Proposition 11.2]{Iitaka} because  $f^{-1}(C)$ is a normal crossing divisor.

Proposition~\ref{prop:main} applied to the divisors $D_i$ then gives
\[
 (\det f^*\shf \cdot D_i) \dleq \max\st{c_2(f^*\shf),0}
\]
for all $1\leq i\leq m$. Summing them up and noting that $f^*D=\sum_{i}D_i$ we obtain
\[
(\det f^*\shf \cdot f^*D) \equ  (\det f^*\shf \cdot \sum_{i=1}^{m}D_i) \dleq \max\st{mc_2(f^*\shf),0}.
\]
If $d$ is the degree of the covering $f$, then
\[
(\det f^*\shf \cdot f^*D) \equ d\cdot (\det\shf\cdot D) \ \ \text{ and } \ \ c_2(f^*\shf)\equ d\cdot c_2(\shf),
\]
which completes the proof.
\end{proof}

\begin{theorem}[Logarithmic Miyaoka--Yau inequality]\label{thm:logMY}
Let $X$ be a smooth complex projective surface, $C$ a semi-stable curve on $X$ such that $K_X+C$ is big. Then
\[
 c_1(\olc)^2 \dleq 3c_2(\olc),
\]
in other words,
\[
 (K_X+C)^2 \dleq 3(e(X)-e(C)).
\]
\end{theorem}
\begin{remark}\rm
We recall that  $\det \olc\simeq \OO_X(K_X+C)$ and $c_2(\olc)=e(X-C)$, where $e$ denotes the topological Euler characteristic (see \cite[Proposition 7.1]{Sakai}).
\end{remark}
\begin{remark}\rm
 Let us quickly recall the definition of a semi-stable curve used here (following \cite{Sakai}). A reduced  $1$-cycle $C\subseteq X$ is called \emph{semi-stable} if
\begin{enumerate}
   \item $C$ is a normal crossing divisor (i.e., has only ordinary nodes as singularities),
   \item each smooth rational component of $C$ intersects the other components of $C$ in at least two points.
\end{enumerate}
It is important to point out that an irreducible semi-stable curve can have ordinary nodes as singularities. If $C$ were defined to be simple normal crossing, instead of just normal crossing, then $C$ irreducible would imply $C$ smooth.
\end{remark}
\begin{remark}\label{rmk:c-exc}\rm
With the notation of Theorem \ref{thm:logMY}, we will consider the following version of minimality (which could reasonably be called log-minimal). A smooth rational curve $E\subseteq X$ is called \emph{$C$-exceptional}, if
\[
 (E^2) \equ -1\ ,\ (C\cdot E)\dleq 1\ .
\]
As outlined on \cite[pp. 1--2]{Sakai}, one can get rid of $C$-exceptional curves by blowing them down.
\end{remark}

\begin{proof}[Proof of Theorem~\ref{thm:logMY}]
According to Sakai's hint \cite[Last sentence in the proof of Theorem 7.6]{Sakai} we will follow Miyaoka's original proof for $\Omega_X^1$.

Our plan is to apply Theorem~\ref{thm:3} to $\shf=\olc$ itself. Exploiting Remark~\ref{rmk:c-exc}, we may and will assume without loss of generality that $X$ contains no $C$-exceptional curves. As proven in \cite[Theorem 5.8]{Sakai}\footnote{Note that it is implicitly assumed in the proof of this result that $X$ is log-minimal, this is why we had to get rid of the $C$-exceptional curves first.}, $K_X+C$ is semi-ample provided it is big, hence $K_X+C$ is nef.  Consequently, the determinant of $\olc$ is nef. We still have to fight for the existence of a non-trivial section of a suitable vector bundle.

   If $c_1(\olc)^2 \dleq 2c_2(\olc)$, then we are done.
   Hence we may assume
\[
 c_1(\olc)^2 > 2c_2(\olc),
\]
and set
\[
 \alpha := \frac{c_2(\olc)}{c_1(\olc)^2} \,<\, \frac{1}{2}.
\]
Here we used that $\alpha>0$, as shown in \cite[Section 7]{Sakai}.
Fix a sufficiently small  number $\delta\in\Q_{\geq 0}$.  Observe that for sufficiently
divisible $m>0$ we have
$$\begin{array}{rcl}
(\det\olc\cdot (m(\alpha+\delta)(K_X+C))) & \equ & m(\alpha+\delta)((K_X+C)^2) \\
  & \equ & m((\alpha+\delta)(c_1(\olc)^2)\\
  & > & m\cdot c_2(\olc).
  \end{array}
$$
By Theorem~\ref{thm:3}, this means that
\[
 \HH{0}{X}{\Sym^m\olc\otimes\OO_X(-m(\alpha+\delta)(K_X+C))} \equ 0
\]
for all large $m$ such that $m(\alpha+\delta)\in\Z$. We next compute $H^2$ of the same vector bundle, with the help of Serre duality:
\begin{eqnarray*}
&& \HH{2}{X}{\Sym^m\olc\otimes\OO_X(-m(\alpha+\delta)(K_X+C))} \\
= && \HH{0}{X}{\OO_X(K_X)\otimes (\Sym^m\olc\otimes\OO_X(-m(\alpha+\delta)(K_X+C)))^*} \\
= && \HH{0}{X}{(\Sym^m\olc)^*\otimes\OO_X(K_X+m(\alpha+\delta)(K_X+C))} \\
= && H^{0}(X,(\Sym^m\olc\otimes\OO_X(-m(K_X+C)))\otimes \\
&&  \OO_X(K_X+m(\alpha+\delta)(K_X+C))) \\
= && \HH{0}{X}{\Sym^m\olc\otimes\OO_X(-m(1-\alpha-\delta)(K_X+C)+K_X)},
\end{eqnarray*}
where we used Remark~\ref{rmk:dual} for the third equality. Next we apply Theorem~\ref{thm:3} to the rank $2$ vector bundle $\shf=\olc$
and the divisor $D=m(1-\alpha-\delta)(K_X+C)-K_X$. As noted above, the determinant of  $\shf=\olc$ is nef and
\begin{eqnarray*}
&& (\det\shf\cdot (m(1-\alpha-\delta)(K_X+C)-K_X)) \\
&& \hspace*{2cm} \equ ((K_X+C)\cdot (m(1-\alpha-\delta)(K_X+C)-K_X))  \\
&& \hspace*{2cm}\equ m(1-\alpha-\delta)((K_X+C)^2)-((K_X+C)\cdot K_X) \\
&& \hspace*{2cm} >  m\cdot c_2(\shf)
\end{eqnarray*}
for all $m$ sufficiently large for which $m(1-\alpha-\delta)\in\Z$. Here we have used
the observation that
\[
 1-\alpha-\delta >0,
\]
which follows from having $\alpha<\tfrac{1}{2}$ and having $\delta$ very small. In fact, we have
$$
   1-\alpha-\delta\geq \alpha+\delta
$$
 and therefore Theorem~\ref{thm:3} implies that
\[
 \HH{0}{X}{\Sym^m\olc\otimes\OO_X(-m(1-\alpha-\delta)(K_X+C)+K_X)} \equ 0.
\]
According to our computations using Serre duality, this is equivalent to
\[
\HH{2}{X}{\Sym^m\olc\otimes\OO_X(-m(\alpha+\delta)(K_X+C))} \equ 0.
\]
Consequently,
\[
 \chi(X,\Sym^m\olc\otimes\OO_X(-m(\alpha+\delta)(K_X+C))) \dleq 0.
\]
On the other hand, the dictionary between the projective bundle
\[
\pi:V:= \P(\olc)\to X
\]
and $X$ (see \cite{Mi}, for example) gives
\begin{eqnarray*}
&&  \chi(X,\Sym^m\olc\otimes\OO_X(-m(\alpha+\delta)(K_X+C)))  \\
&&\ \ \ \ \ \ \ \ \ \ \ \ \ \ \ \ \ \ \ \ \equ \chi(V,\OO_V(m(H-(\alpha+\delta)\pi^*(K_X+C)))),
\end{eqnarray*}
where the right hand side grows with $m$ asymptotically as
\[
\frac{1}{6}(H-(\alpha+\delta)\pi^*(K_X+C)))^3m^3
\]
by the asymptotic Riemann--Roch Theorem. As a result, we obtain
\[
 (H-(\alpha+\delta)\pi^*(K_X+C)))^3 \dleq 0.
\]
By letting $\delta\to 0$, we arrive at
\[
 (H-\alpha\pi^*(K_X+C))^3 \dleq 0.
\]
To evaluate this, we recall the various intersection numbers on $V$ (see \cite[Lemma 5]{Mi}):
\[
(H^3) \equ c_1(\shf)^2-c_2(\shf)\ ,\ (H^2\cdot\pi^*D) \equ (\det\shf\cdot D)\ ,\ (H\cdot\pi^*D\cdot\pi^*D') \equ (D\cdot D').
\]
Therefore,
\begin{eqnarray*}
 (H-\alpha\pi^*(K_X+C))^3 & =  &(H^3)-3(H^2\cdot \alpha(\pi^*(K_X+C))) \\
&   +  & 3(H\cdot(\alpha\pi^*(K_X+C))^2)-((\alpha\pi^*(K_X+C))^3) \\
& =  &(c_1(\olc)^2-c_2(\olc)) \\
& - & 3\alpha(c_1(\olc)\cdot (K_X+C)) \\
& + & 3\alpha^2((K_X+C)^2) + 0.
\end{eqnarray*}
Since $c_1(\olc)=\OO_X(K_X+C)$ and $c_2(\olc)=\alpha c_1(\olc)^2$ (this latter by the definition of $\alpha$), after collecting terms we get
\begin{eqnarray*}
0 \dgeq (H-\alpha\pi^*(K_X+C))^3 & \equ & (1-4\alpha+3\alpha^2)c_1(\olc)^2 \\
& = & (1-\alpha)(1-3\alpha)c_1(\olc)^2.
\end{eqnarray*}
As $0\leq \alpha<\tfrac{1}{2}$ and $c_1(\olc)>0$ (since $K_X+C$ is big and nef), we conclude that $3\alpha\geq 1$, which is what we wanted to show.
\end{proof}

\paragraph*{Acknowledgement.} We would like to thank the Mathematisches Forschungsinstitut
Oberwolfach for providing a perfect venue for our workshop.
We also thank Lawrence Ein, Laurent Evain, Jun-Muk Hwang and
Roberto Mu\~{n}oz, whose contributions to these notes through their active participation during the workshop
we gratefully acknowledge, and Nick Shepherd-Barron for his comments
regarding the bounded negativity problem. Finally, we are grateful to Stefan Kebekus for suggesting
the use of subversion for helping to make possible the preparation of an article with so
many coauthors and for his mathematical comments and support for the project.

\bigskip
\small
   Tho\-mas Bau\-er,
   Fach\-be\-reich Ma\-the\-ma\-tik und In\-for\-ma\-tik,
   Philipps-Uni\-ver\-si\-t\"at Mar\-burg,
   Hans-Meer\-wein-Stra{\ss}e,
   D-35032~Mar\-burg, Germany.

\nopagebreak
   \textit{E-mail address:} \texttt{tbauer@mathematik.uni-marburg.de}

\bigskip
   Cristiano Bocci,
   Dipartimento di Scienze Matematiche e Informatiche ``R. Magari",
   Universit\`a di Siena,
   Pian dei Mantellini, 44
   53100 Siena, Italy.

\nopagebreak
   \textit{E-mail address:} \texttt{bocci24@unisi.it}

\bigskip
   Susan Cooper,
   Department of Mathematics,
   University of Nebraska-Lincoln,
   Lincoln, NE 68588-0130, USA.

\nopagebreak
   \textit{E-mail address:} \texttt{scooper4@math.unl.edu}

\bigskip
   Sandra Di Rocco,
   Department of Mathematics,
   Royal Institute of Technology (KTH),
   S-10044 Stockholm, Sweden.

\nopagebreak
   \textit{E-mail address:} \texttt{dirocco@math.kth.se}

\bigskip
Marcin Dumnicki,
Institute of Mathematics,
Jagiellonian University
ul. \L{}ojasiewicza 6,
30-348 Krak\'ow,
Poland.

\nopagebreak
\textit{E-mail address:} \texttt{Marcin.Dumnicki@im.uj.edu.pl}

\bigskip
   Brian Harbourne,
   Department of Mathematics,
   University of Nebraska-Lincoln,
   Lincoln, NE 68588-0130, USA.

\nopagebreak
   \textit{E-mail address:} \texttt{bharbour@math.unl.edu}

\bigskip
   Kelly Jabbusch,
   Department of Mathematics,
   Royal Institute of Technology (KTH),
   S-10044 Stockholm, Sweden.

\nopagebreak
   \textit{E-mail address:} \texttt{jabbusch@math.kth.se}

\bigskip
Andreas Leopold Knutsen,
Department of Mathematics,
University of Bergen,
Johs. Brunsgt. 12,
N-5008 Bergen, Norway.

\nopagebreak
\textit{E-mail address:} \texttt{andreas.knutsen@math.uib.no}

\bigskip
   Alex K\"uronya,
   Budapest University of Technology and Economics,
   Mathematical Institute, Department of Algebra,
   Pf. 91, H-1521 Budapest, Hungary.

\nopagebreak
   \textit{E-mail address:} \texttt{alex.kuronya@math.bme.hu}

\medskip
   \textit{Current address:}
   Alex K\"uronya,
   Albert-Ludwigs-Universit\"at Freiburg,
   Mathematisches Institut,
   Eckerstra{\ss}e 1,
   D-79104 Freiburg,
   Germany.

\bigskip
    Rick Miranda,
    Department of Mathematics, Colorado State University,
    Fort Collins, CO  80523,  USA.

\nopagebreak
    \textit{E-mail address:} \texttt{rick.miranda@colostate.edu}
	
\bigskip
   Joaquim Ro\'e,
   Departament de Matem\`atiques, Universitat Aut\`onoma de Barcelona
   E-08193 Bellaterra (Barcelona),
   Spain.

\nopagebreak
   \textit{E-mail address:} \texttt{jroe@mat.uab.cat}

\bigskip
Hal Schenck,
Mathematics Department, Univ. Illinois,
Urbana, Illinois, 61801 USA.

\nopagebreak
\textit{E-mail address:} \texttt{schenck@math.uiuc.edu}

\bigskip
   Tomasz Szemberg,
   Instytut Matematyki UP,
   PL-30-084 Krak\'ow, Poland.

\nopagebreak
   \textit{E-mail address:} \texttt{tomasz.szemberg@uni-due.de}

\medskip
   \textit{Current address:}
   Tomasz Szemberg,
   Albert-Ludwigs-Universit\"at Freiburg,
   Mathematisches Institut,
   Eckerstra{\ss}e 1,
   D-79104 Freiburg,
   Germany.

\bigskip
  Zach Teitler,
  Department of Mathematics,
  Boise State University,
  1910 University Drive,
  Boise, ID 83725-1555,
  USA.

\nopagebreak
  \textit{E-mail address:} \texttt{zteitler@boisestate.edu}

\end{document}